\documentclass[10 pt, leqno]{article}
\date{}
\usepackage{amssymb, amsbsy, amsmath, amsfonts, amssymb, amscd, mathrsfs, amsthm, dsfont}

\usepackage[english]{babel}
\usepackage[T1]{fontenc}
\usepackage{indentfirst}
\usepackage{color}

\makeatletter
\@addtoreset{equation}{section}
\makeatother

\newtheorem{statement}{}[section]
\newtheorem{theorem}[statement]{Theorem}
\newtheorem{lemma}[statement]{Lemma}

\newtheorem{proposition}[statement]{Proposition}
\newtheorem{definition}[statement]{Definition}
\newtheorem{corollary}[statement]{Corollary}

\newcommand\C{\mathbb C}

\newcommand\R{\mathbb R}
\newcommand\T{\mathbb T}
\newcommand\D{\mathbb D}

\newcommand\e{{\rm e}}

\newcommand\eps{\varepsilon}
\newcommand\ind{\mathds{1}}
\newcommand\dis{\displaystyle}
\renewcommand \Re{{\mathfrak R}{\rm e}\,}

\newcommand\converge{\mathop{\longrightarrow}\limits}

\let\phi=\varphi
\let\hat = \widehat
\let\tilde=\widetilde
\newcommand\tq{\, ; \ }
\newcommand{\overbar}[1]{\mkern 1.5mu\overline{\mkern-1.5mu#1\mkern-1.5mu}\mkern 1.5mu}

\title{\bf Boundedness of composition operators on general weighted Hardy spaces of analytic functions}
\author{\it Pascal~Lef\`evre, Daniel~Li,  \\ \it Herv\'e~Queff\'elec, Luis~Rodr{\'\i}guez-Piazza}

\date{\footnotesize \today}

\begin{document}

\maketitle

\noindent {\bf Abstract.} We characterize the (essentially) decreasing sequences of positive numbers $\beta = (\beta_n)$ for which all composition operators on 
$H^2 (\beta)$ are bounded, where $H^2 (\beta)$ is the space of analytic functions $f$ in the unit disk such that 
$\sum_{n = 0}^\infty |c_n|^2 \beta_n < \infty$ if $f (z) = \sum_{n = 0}^\infty c_n z^n$. We also give conditions for the boundedness when $\beta$ is not 
assumed essentially decreasing.
\medskip

\smallskip


\section {Introduction} \label{sec: intro}

Let $\beta = (\beta_n)_{n\geq 0}$ be a sequence of positive numbers such that
\begin{equation} \label{unus} 
\liminf_{n \to \infty} \beta_n^{1/n} \geq 1 \, .
\end{equation} 
The associated weighted Hardy space $H^{2} (\beta)$ is the set of analytic functions $f (z) = \sum_{n = 0}^\infty a_n z^n$  such that 
 \begin{equation} \label{duo} 
\Vert f \Vert^2 := \sum_{n = 0}^\infty |a_n|^2 \beta_n < \infty \, . 
\end{equation}
It is a Hilbert space of analytic functions on $\D$. When $\beta_n \equiv 1$, we recover the usual Hardy space $H^2$.
\smallskip

Recall that a symbol is a (non constant) analytic self-map $\phi \colon \D \to \D$, and the associated composition operator 
$C_\phi \colon H^2 (\beta) \to {\cal H}ol (\D)$ is defined (formally) as:
\begin{equation} 
C_\phi (f) = f \circ \phi \, .
\end{equation} 

An important question in the theory is to decide when $C_\phi$ is bounded on $H^2 (\beta)$, i.e. when  $C_\phi \colon H^2 (\beta) \to H^2 (\beta)$. 
When $H^2 (\beta)$ is the usual Hardy space $H^2$ (i.e. when $\beta_n \equiv 1$), it is well-known (\cite[pp.~13--17]{Shapiro-livre}) that all symbols generate 
bounded composition operators. On the other hand, for the Dirichlet space, corresponding to $\beta_n = n + 1$, not all composition operators are bounded. 
Note that, by definition of the norm of $H^2 (\beta)$, all rotations $R_\theta$, $\theta \in \R$, induce bounded composition operators on $H^2 (\beta)$ 
and send isometrically $H^2 (\beta)$ into itself.
\smallskip

Our goal in this paper is characterizing the (non-increasing) sequences $\beta$ for which all composition operators act boundedly on the space $H^2 (\beta)$, i.e. 
send $H^2 (\beta)$ into itself. We will prove the following result, where $T_a (z) = \frac{a + z}{1 + \bar{a} \, z}$ for $a, z \in \D$.

\begin{theorem} \label{main theorem completed}
Let $\beta$ be an essentially decreasing sequence of positive numbers. The following assertions are equivalent:

\begin{itemize}
\setlength\itemsep {-0.1 em}

\item [$1)$] all composition operators are bounded on $H^2 (\beta)$;

\item [$2)$] all maps $T_a$, for $0 < a < 1$, induce bounded composition operators $C_{T_a}$ on $H^2 (\beta)$;

\item [$3)$] for some $a \in (0, 1)$, the map $T_a$ induces a bounded composition operator $C_{T_a}$ on $H^2 (\beta)$;

\item [$4)$] $\beta$ satisfies the $\Delta_2$-condition.
\end{itemize}
\end{theorem}

For the notion of essentially decreasing sequence, see Definition~\ref{def: essentially decreasing}; the $\Delta_2$-condition is defined in \eqref{Delta-deux}.
Actually, we will obtain Theorem~\ref{main theorem completed} by gluing Theorem~\ref{main theorem} and Theorem~\ref{theo necessity Delta-deux}. 
\medskip

Most of the existing works with weighted Hardy spaces concern the case 
\begin{equation} \label{integral beta}
\beta_n = \int_{0}^1 t^n \, d \sigma (t) 
\end{equation} 
where $\sigma$ is a positive measure on $(0, 1)$. More specifically the following definition is often used. Let $G \colon (0, 1) \to \R_+$ be an integrable function 
and let $H^2_G$ be the space of analytic functions $f \colon \D \to \C$ such that:
\begin{equation} \label{particular case}
\Vert f \Vert_{H^2_G}^{2} := \int_{\D} |f (z)|^2 \, G (1 - |z|^2) \, dA (z) < \infty \, .
\end{equation} 

Such weighted Bergman type spaces are used, for instance, in \cite{Kriete-MacCluer}, \cite{Karim-Pascal} and in \cite{LQR-radius}. We have 
$H^2_G = H^2 (\beta)$ with:
\begin{equation} \label{coto}
\beta_n = 2 \int_{0}^{1} r^{2 n + 1} G (1 - r^2) \, dr = \int_{0}^{1} t^{n} \, G (1 - t ) \, dt \, ,
\end{equation} 
and the sequence $\beta = (\beta_n)_n$ is non-increasing. 

In Shapiro's presentation, the main point is the case $\phi (0) = 0$ and a subordination principle for subharmonic functions (Littlewood's subordination principle). 
The case of automorphisms is claimed simple, using an integral representation for the norm and some change of variable.  When $\beta$ is defined as in 
\eqref{integral beta}, one disposes of integral representations for the norm in $H^2 (\beta)$, and, as in the Hardy space case, this integral representation rather 
easily  gives the boundedness of $C_{T_a}$ on $H^2$, where 
\begin{equation} 
T_ a (z) = \frac{a + z}{1+ \bar{a} \, z} 
\end{equation} 
for $a \in \D$. But the above representation \eqref{integral beta} is equivalent, by the Hausdorff moment theorem, to a high regularity of the sequence $\beta$, 
namely its \emph{complete monotony}. When integral representations fail, we have to work with bare hands. If  the symbol vanishes at the origin, 
Kacnel'son's theorem gives a positive answer when $\beta$ is essentially decreasing (see \cite{Isabelle} or \cite[Theorem~3.12]{LLQR-comparison}). Actually 
that follows from an older theorem of Goluzin \cite {Goluzin} (see \cite[Theorem~6.3]{Duren}), which itself uses a refinement by Rogosinski of Littlewood's 
principle (\cite[Theorem~6.2]{Duren}). So that the main issue remains  the boundedness of $C_{T_a}$. 

A polynomial minoration (see Definition~\ref{def: polynomial minoration} below) for $\beta$ is necessary for any $C_{T_a}$ to be bounded on $H^2 (\beta)$ 
(Proposition~\ref{polynomial lb necessary}) and we showed in \cite[end of Section~3]{LLQR-comparison} that for $\beta_n = \exp (- \sqrt{n})$, $C_{T_a}$ is 
never bounded on $H^{2} (\beta)$. But this polynomial minoration is not sufficient, as we will see in Theorem~\ref{main theorem completed} below, which also 
evidences the basic role of the maps $T_a$ in the question.
\smallskip

However, we construct a weight $\beta$ which is not essentially decreasing and for which all composition operators with symbol vanishing at $0$ are 
bounded (Theorem~\ref{theo example}), though no map $T_a$ with $0 < a < 1$ induces a bounded composition operator 
(Proposition~\ref{T_a not bounded}). 
\medskip

For spaces of Bergman type $A^2_{\tilde G} := H^2_G$, where $\tilde G (r) = G (1 - r^2)$, defined as the spaces of analytic functions in $\D$ such that 
$\int_\D | f (z)|^2 \, {\tilde G} (|z|) \, dA < \infty$, for a positive non-increasing continuous function ${\tilde G}$ on $[0, 1)$, Kriete and MacCluer studied in 
\cite{Kriete-MacCluer} some analogous problems. They proved, in particular \cite[Theorem~3]{Kriete-MacCluer} that, for:
\begin{displaymath} 
\qquad \qquad \qquad {\tilde G} (r) = \exp \, \bigg( - B \, \frac{1}{(1 - r)^\alpha} \bigg) \, , \qquad B > 0 \, , \ 0 < \alpha \leq 2 \, ,
\end{displaymath} 
and 
\begin{displaymath} 
\qquad \qquad \qquad \qquad \phi (z) = z + t (1 - z)^\beta \, , \qquad \qquad 1 <\beta \leq 3 \, , \ 0 < t < 2^{1 - \beta} \, ,
\end{displaymath} 
then $C_\phi$ is bounded on $A^2_{\tilde G}$ if and only if $\beta \geq \alpha + 1$.

Here
\begin{displaymath} 
\beta_n = \int_0^1 t^n \e^{- B / (1 - \sqrt{t})^\alpha} \, dt \, ;
\end{displaymath} 
and, since $\beta_n \approx \exp ( - c \, n^{\alpha / (\alpha  +1)} )$, the sequence $(\beta_n)$ does not satisfy the $\Delta_2$-condition, accordingly to 
our Theorem~\ref{main theorem} below.
\medskip

We end the paper with some miscellaneous remarks.

\section {Definitions, notation, and preliminary results} \label{sec: def}

The open unit disk of $\C$ is denoted $\D$ and we write $\T$ its boundary $\partial \D$. We set $e_n (z) = z^n$, $n \geq 0$.
\smallskip

The weighted Hardy space $H^2 (\beta)$ defined in the introduction is a Hilbert space with the canonical orthonormal basis 
\begin{equation} 
\qquad \qquad e^\beta_{n} (z) = \frac{1}{\sqrt{\beta_n}} \, z^n \, , \quad  n\geq 0 \, ,
\end{equation} 
and the reproducing kernel $K_w$ given for all $w\in \D$ by 
 \begin{equation} \label{tres} 
K_{w} (z) = \sum_{n = 0}^\infty e^\beta_{n} (z) \, \overbar{e^\beta_{n} (w)} 
= \sum_{n = 0}^\infty \frac{1}{\beta_n} \,\overbar{w}^n \, z^n \, .
\end{equation}

Note that \eqref{unus} is necessary for $H^2 (\beta)$ to consist of analytic functions in $\D$. Indeed the fact that 
$\sum_{n \geq 1} \frac{1}{n \sqrt{\beta_n}} \, z^n$ belongs to $H^2 (\beta)$ and is analytic in $\D$ implies \eqref{unus}.
Note also that $H^2$ is continuously embedded in $H^2 (\beta)$ if and only if $\beta$ is bounded above. In particular, this is the case when $\beta$ is 
non-increasing. In this paper, we need a slightly more general notion.

\begin{definition} \label{def: essentially decreasing}
A sequence of positive numbers $\beta = (\beta_n)$ is said \emph{essentially decreasing} if, for some constant $C \geq 1$, we have, for all $m \geq n \geq 0$:
\begin{equation} 
\beta_m \leq C \, \beta_n \, .
\end{equation} 
\end{definition}

Note that saying that $\beta$ is essentially decreasing means that the shift operator on $H^2 (\beta)$ is power bounded.
\smallskip

If $\beta$ is essentially decreasing, and if we set:
\begin{displaymath} 
\tilde \beta_n = \sup_{m \geq n} \beta_m \, ,
\end{displaymath} 
the sequence $\tilde \beta = (\tilde \beta_n)$ is non-increasing and we have $\beta_n \leq \tilde \beta_n \leq C \, \beta_n$. In particular, the space 
$H^2 (\beta)$ is isomorphic to $H^2 (\tilde \beta)$ and $H^2$ is continuously embedded in $H^2 (\beta)$. 

\begin{definition} \label{def: Delta-deux}
The sequence of positive numbers $\beta = (\beta_n)$ is said to satisfy the \emph{$\Delta_2$-condition} if there is a positive constant $\delta < 1$ such that,  
for all integers $n \geq 0$: 
\begin{equation} \label{Delta-deux}  
\beta_{2 n} \geq \delta \, \beta_n \, .
\end{equation}
\end{definition}

This terminology is given by analogy with that used for Orlicz functions. 

\begin{definition} \label{def: polynomial minoration}
The sequence of positive numbers $\beta = (\beta_n)$ is said to have a \emph{polynomial minoration} if  there are positive constants $\delta$ and $\alpha$ 
such that, for all integers $n \geq 1$:
\begin{equation}\label{slow}  
\beta_{n} \geq \delta \, n^{- \alpha}.
\end{equation} 
\end{definition}
That means that $H^2 (\beta)$ is continuously embedded in the weighted Bergman space ${\mathfrak B}^2_\alpha$ of the analytic functions 
$f \colon \D \to \C$ such that 
\begin{displaymath} 
\| f \|_{{\mathfrak B}^2_\alpha}^2 := (\alpha + 1) \int_\D |f (z)|^2 (1 - |z|^2)^\alpha \, dA (z) < \infty
\end{displaymath} 
since ${\mathfrak B}^2_\alpha = H^2 (\gamma)$ with $\gamma_n \approx n^{- \alpha}$.
\medskip

The following simple proposition links those notions.

\begin{proposition} \label{simple proposition polynomial minoration}
Let $\beta$ be an essentially decreasing sequence of positive numbers. Then if $\beta$ satisfies the $\Delta_2$-condition, it has a polynomial minoration.

The converse does not hold.
\end{proposition}
\begin{proof} 
Assume $\beta_m \leq C \, \beta_n$ for $m \geq n$ and $\beta_{2 p} \geq \e^{- A} \,  \beta_p$. Let now $n$ be an  integer 
$\geq 2$, and $k \geq 1$ the smallest integer such that $2^k \geq n$, so that $k \leq a \log n$ with $a$ a positive constant. We get:
\begin{displaymath} 
\beta_n \geq C^{- 1} \beta_{2^k} \geq C^{- 1} \e^{- k A} \beta_1 \geq C^{- 1} \beta_1 \e^{- a A \log n} =: \rho \, n^{- \alpha} \, ,
\end{displaymath} 
with $\rho = C^{- 1} \beta_1$ and $\alpha = a A$. 
\smallskip

Let us now see that the converse does not hold. Let $\delta > 0$. We set $\beta_0 = \beta_1 = 1$ and for $n \geq 2$:
\begin{displaymath} 
\beta_n = \frac{1}{(k!)^\delta} \quad \text{when } k! < n \leq (k + 1)! \, .
\end{displaymath} 
The sequence $\beta$ is non-increasing.

For $n$ and $k$ as above, we have:
\begin{displaymath} 
\beta_n = \frac{1}{(k!)^\delta} \geq  \frac{1}{n^\delta} \, ;
\end{displaymath} 
hence $\beta$ has arbitrary polynomial minoration. However we have, for $k \geq 2$:
\begin{displaymath} 
\frac{\beta_{2 (k!)}}{\beta_{k!}} = \frac{(k!)^{- \delta}}{[(k - 1)!]^{- \delta}} = \frac{1}{k^{\delta}} \converge_{k \to \infty} 0 \, ,
\end{displaymath} 
so $\beta$ fails to satisfy the $\Delta_2$-condition.
\end{proof}
\smallskip

For $a \in \D$, we define:
\begin{equation} 
\qquad \quad T_a (z) = \frac{a + z}{1 + \bar{a} \, z} \, \raise 1 pt \hbox{,} \quad z \in \D \, .
\end{equation} 

Recall that $T_a$ is an automorphism of $\D$ and that $T_a (0) = a$ and $T_a (- a) = 0$.

Though we do not need this, we may remark that 
$(T_a)_{a \in (- 1, 1)}$ is a group and $(T_a)_{a \in (0, 1)}$ is a semigroup. 
It suffices to see that $T_a \circ T_b = T_{a \ast b}$, with:
\begin{equation} \label{produit}
a \ast b = \frac{a + b}{1 + a b} \, \cdot 
\end{equation} 
\goodbreak

\begin{proposition} \label{polynomial lb necessary}
Let $a \in (0, 1)$ and assume that $T_a$ induces a bounded composition operator on $H^2 (\beta)$. Then $\beta$ has a polynomial minoration.
\end{proposition}
\begin{proof} 
Since 
\begin{displaymath} 
\| K_x \|^2 = \sum_{n = 0}^\infty \frac{x^{ 2 n}}{\beta_n} \,  \raise 1 pt \hbox{,}
\end{displaymath} 
we have $\| K_x \| \leq \| K_y \|$ for $0 \leq x \leq y < 1$.
\smallskip

We define by induction a sequence $(u_n)_{n \geq 0}$ with:
\begin{displaymath}
u_0 = 0 \qquad \text{and} \qquad u_{n + 1} = T_a (u_n) \, .
\end{displaymath}
Since $T_a (1) = 1$ (recall that $a \in (0, 1)$), we have:
\begin{displaymath}
1 - u_{n + 1} = \int_{u_n}^1 T_a ' (t) \, dt = \int_{u_n}^1 \frac{1 - a^2}{(1 + a t)^2} \, dt \, ;
\end{displaymath}
hence 
\begin{displaymath}
\frac{1 - a}{1 + a} \, (1 - u_n) \leq 1 - u_{n + 1} \leq (1 - a^2) (1 - u_n) \, .
\end{displaymath}
Let $0 < x < 1$. We can find $N \geq 0$ such that $u_N \leq x < u_{N  + 1}$. Then:
\begin{displaymath}
1 - x \leq 1 - u_N \leq (1 - a^2)^N \, .
\end{displaymath}

On the other hand, since $C_{T_a}^{^{\, \ast}} K_z = K_{T_a (z)}$ for all $z \in \D$, we have:
\begin{displaymath}
\| K_x \| \leq \| K_{u_{N  + 1}} \| \leq \| C_{T_a} \| \, \| K_{u_N} \| \leq \| C_{T_a} \|^{N  + 1} \| K_{u_0} \| = \| C_{T_a} \|^{N  + 1} \, .
\end{displaymath}
Let $ s \geq 0$ such that $(1 - a^2)^{- s} = \| C_{T_a} \|$. We obtain:
\begin{equation} \label{slow growth of reproducing kernels}
\| K_x \| \leq \| C_{T_a} \| \, \frac{1}{(1 - x)^s} \, \cdot
\end{equation}
But
\begin{displaymath}
\| K_x \|^2 = \sum_{k = 0}^\infty \frac{x^{2 k}}{\beta_k} \, ;
\end{displaymath}
so we get, for any $k \geq 2$:
\begin{displaymath}
\frac{x^{2 k}}{\beta_k} \leq  \| C_{T_a} \|^2 \frac{1}{(1 - x)^{2 s}} \, \cdot
\end{displaymath}
Taking $x = 1 - \frac{1}{k}\,$, we obtain $\beta_k \geq C \, k^{- 2 s}$. 
\end{proof}

\noindent {\bf Remarks.} 1) For example, when $\beta_n = \exp \big[ - c \, \big( \log (n + 1) \big)^2 \big]$, with $c > 0$, no $T_a$ induces a bounded 
composition operator on $H^2 (\beta)$, though all symbols $\phi$ with $\phi (0) = 0$ are bounded, since $\beta$ is decreasing, as we will see in 
Proposition~\ref{phi(0)=0}.
\smallskip

2) For the Dirichlet space ${\cal D}^2$, we have $\beta_n = n + 1$, but all the maps $T_a$ induce bounded composition operators on ${\cal D}^2$ 
(see \cite[Remark before Theorem~3.12]{LLQR-comparison}). In this case $\beta$ has a polynomial minoration though it is not bounded above. 
\smallskip

3) However, even for decreasing sequences, a polynomial minoration for $\beta$ is not enough for some $T_a$ to induce a bounded composition operator. 
Indeed, we saw in Proposition~\ref{simple proposition polynomial minoration} an example of a decreasing sequence $\beta$  with polynomial minoration, 
but not sharing the $\Delta_2$-condition, and we will see in Theorem~\ref{theo necessity Delta-deux} that the $\Delta_2$-condition is needed for having 
some $T_a$ inducing a bounded composition operator.
\smallskip

4) In \cite{Eva-Jonathan}, Eva Gallardo-Guti\'errez and Jonathan Partington give estimates for the norm of $C_{T_a}$, with $a \in (0, 1)$, when $C_{T_a}$ is 
bounded on $H^2 (\beta)$. More precisely, they proved that if $\beta$ is bounded above and $C_{T_a}$ is bounded, then
\begin{displaymath} 
\| C_{T_a} \| \geq \bigg( \frac{1 + a}{1 - a} \bigg)^\sigma \, \raise 1 pt \hbox{,}
\end{displaymath} 
where $\sigma = \inf \{s \geq 0 \tq (1 - z)^{- s} \notin H^2 (\beta) \}$, and 
\begin{displaymath} 
\| C_{T_a} \| \leq \bigg( \frac{1 + a}{1 - a} \bigg)^\tau \, \raise 1 pt \hbox{,}
\end{displaymath} 
where $\tau = \frac{1}{2} \, \sup \Re W (A)$, with $A$ the infinitesimal generator of the continuous semigroup $(S_t)$ defined as $S_t = C_{T_{\tanh t}}$, 
namely $(Af) (z) = f ' (z) (1 - z^2)$, and $W (A)$ its numerical range. 

For $\beta_n = 1 / (n + 1)^\nu$ with $0 \leq \nu \leq 1$, the two bounds coincide, so they get 
$\| C_{T_a} \| = \big( \frac{1 + a}{1 - a} \big)^{(\nu + 1)/2}$. 
\smallskip

5) We saw in the proof of Proposition~\ref{polynomial lb necessary} that if $C_{T_a}$ is bounded on $H^2 (\beta)$ for some 
$a \in (0, 1)$, then the reproducing  kernels $K_w$ have, by \eqref{slow growth of reproducing kernels}, a \emph{slow growth}:
\begin{equation} 
\| K_w \| \leq \frac{C}{(1 - |w|)^s} 
\end{equation} 
for positive constants $C$ and $s$. Actually, we have the following equivalence.

\begin{proposition}
The sequence $\beta$ has a polynomial minoration if and only if the reproducing kernels $K_w$ of $H^2 (\beta)$ have a slow growth.
\end{proposition}
\begin{proof}
The sufficiency is easy and seen at the end of the proof of Proposition~\ref{polynomial lb necessary}. For the necessity, we only have to see that:
\begin{displaymath} 
\| K_w \|^2 = \frac{1}{\beta_0} + \sum_{n = 1}^\infty \frac{\, \, \, |w|^{2 n}}{\beta_n} 
\leq \frac{1}{\beta_0} + \delta^{- 1} \sum_{n = 1}^\infty n^\alpha |w|^{2 n} \leq \frac{C}{(1 - |w|^2)^{\alpha  + 1}} \, \cdot \qedhere
\end{displaymath} 
\end{proof}
\goodbreak

\section {Boundedness of composition operators} \label{sec: boundedness comp op}

We study in this section conditions ensuring that all composition operators on $H^2 (\beta)$ are bounded.

\subsection {Conditions on the weight} \label{sec: weight}

We begin with this simple observation.

\begin{proposition} 
If all composition operators, and even if all composition operators with symbol vanishing at $0$, are bounded on $H^2 (\beta)$, then the sequence 
$\beta$ is bounded above.
\end{proposition}
\begin{proof}
Let $f \in H^\infty$.  Write $f = A \, \phi +  f (0)$ where $A$ is a constant and $\phi$ a symbol vanishing at $0$. We have $\phi = C_\phi (z) \in H^2 (\beta)$, by 
hypothesis. So that $f \in H^2 (\beta)$ and $H^\infty \subseteq H^2 (\beta)$.  It follows (by the closed graph theorem) that there exists a constant $M$ such that
 $\| f \|_{H^2 (\beta)} \leq M \, \| f \|_\infty$ for all $f \in H^\infty$. Testing that with $f (z) = z^n$, we get $\beta_n \leq M^2$. 
\end{proof}

For symbols vanishing at $0$, we have the following characterization. 

\begin{proposition} \label{phi(0)=0}
The following assertions are equivalent:
\smallskip

$1)$ all symbols $\phi$ such that $\phi (0) = 0$ induce bounded composition operators $C_\phi$ on $H^2 (\beta)$ and
\begin{equation} \label{borne uniforme}
\sup_{\phi (0) = 0} \| C_\phi \| < \infty \, ;
\end{equation}

$2)$ $\beta$ is an essentially decreasing sequence.
\end{proposition}

Of course, by the uniform boundedness principle, \eqref{borne uniforme} is equivalent to:
\begin{displaymath} 
\qquad \qquad \sup_{\phi (0) = 0} \| f \circ \phi \| < \infty \quad \text{for all } f \in H^2 (\beta) \, .
\end{displaymath} 
\begin{proof}
$2) \Rightarrow 1)$ We may assume that $\beta$ is non-increasing. Then the Goluzin-Rogosinski theorem (\cite[Theorem 6.3]{Duren}) gives the result; 
in fact, writing $f (z) = \sum_{n = 0}^\infty c_n z^n$ and $(C_\phi f) (z) = \sum_{n = 0}^\infty d_n z^n$, it says that:
\begin{displaymath} 
\| C_\phi f \|^2 = |d_0|^2 \beta_0 + \sum_{n = 1}^\infty |d_n|^2 \beta_n \leq |c_0|^2 \beta_0 + \sum_{n = 1}^\infty |c_n|^2 \beta_n 
= \| f \|^2 \, ,
\end{displaymath} 
leading to $C_\phi$ bounded and $\| C_\phi \| \leq 1$. Alternatively, we can use a result of Kacnel'son (\cite{Kacnelson}; see also 
\cite{Isabelle}, \cite[Corollary 2.2]{IsabelleII}, or \cite[Theorem 3.12]{LLQR-comparison}). This result was also proved by C.~Cowen 
\cite[Corollary of Theorem~7]{Cowen}.

$1) \Rightarrow 2)$ Set $M = \sup_{\phi (0) = 0} \| C_\phi \|$. Let $m > n$, and take:
\begin{displaymath}
\phi (z) = \phi_{m, n} (z) = z \, \bigg( \frac{1 + z^{m - n}}{2} \bigg)^{1 / n} \, .
\end{displaymath}
Then $\phi (0) = 0$ and  $[\phi (z)]^n = \frac{z^n + z^m}{2}\,$; hence
\begin{displaymath} 
\frac{1}{4} \, (\beta_n + \beta_m)  = \| \phi^n \|^2 = \| C_\phi (e_n) \|^2 \leq \| C_\phi \|^2 \| e_n \|^2 \leq M^2 \, \beta_n \, ,
\end{displaymath} 
so $\beta$ is essentially decreasing.
\end{proof}

Boundedness of $\beta_n$ does not suffice. For example, let $(\beta_n)$ be a sequence such that 
 $\beta_{2 k +2} / \beta_{2 k + 1} \converge_{k \to \infty} \infty$ (for instance $\beta_{2 k} = 1$ and 
$\beta_{2 k + 1} = 1 / (k + 1)$); if $\phi (z) = z^2$, then $\| C_\phi (z^{2 n + 1}) \|^2 = \| z^{2 (2 n + 1)} \|^2 = \beta_{2 (2 n + 1)}$; since 
$\| z^{2 n + 1} \|^2 = \beta_{2 n + 1}$, $C_\phi$ is not bounded on $H^2 (\beta)$. 
\smallskip

A more interesting example is the following. For $0 < r < 1$, let $\beta_n =\pi \,  n \, r^{2 n}$. This sequence is eventually decreasing, so it is essentially 
decreasing. The square $\| f \|^2_{H^2 (\beta)}$ of the norm $\| f \|_{H^2 (\beta)}$ is the area of the part of the Riemann surface on which $r \D$ is mapped 
by $f$. E.~Reich \cite{Reich}, generalizing Goluzin's result \cite{Goluzin} (see \cite[Theorem~6.3]{Duren}), proved that for all symbols $\phi$ such that 
$\phi (0) = 0$, the composition operator $C_\phi$ is bounded on $H^2 (\beta)$ and
\begin{displaymath} 
\| C_\phi \| \leq \sup_{n \geq 1} \sqrt{n} \, r^{n - 1} \leq \frac{1}{\sqrt{2 \, \e}} \, \frac{1}{r \sqrt{\log (1 / r)}} \, \cdot
\end{displaymath} 
For $0 < r < 1 / \sqrt{2}$, Goluzin's theorem asserts that $\|C_\phi \| \leq 1$.
\smallskip

Note that this sequence $\beta$ does not satisfy the $\Delta_2$-condition since $\beta_{2 n} / \beta_n = 2 \, r^{2 n}$,  
Theorem~\ref{theo necessity Delta-deux} below states that no composition operator $C_{T_a}$ is bounded. 
\medskip

However that the weight $\beta$ is essentially decreasing is not necessary for the boundedness of all composition operators $C_\phi$,  
with symbol $\phi$ vanishing at $0$, as we will see later (Theorem~\ref{theo example}). 

\subsection {Sufficient condition for the boundedness of composition operators -- Part I} \label{sec: boundedness}

We now have one of the the main results of this section.

\begin{theorem} \label{main theorem}
Let $H^2 (\beta)$ be a weighted Hardy space with $\beta = (\beta_n)$ essentially decreasing and satisfying the $\Delta_2$-condition. Then all composition 
operators on $H^2 (\beta)$ are bounded. 
\end{theorem} 

For the proof, we need a lemma. 
\goodbreak

\begin{lemma} \label{majo coeff}
For $0 < a < 1$, we write:
\begin{equation} \label{T_a power n}
(T_a z)^n = \sum_{m = 0}^\infty a_{m, n} z^m \, .
\end{equation} 
Then, there are constants $b>0,\  0<C_1<1,\  C_2>1$ such that 
\begin{displaymath}
| a_{m, n} | \leq \left\{ 
\begin{array} {lcl}
\e^{- b n} & \text{if} & m \leq C_1 n \, , \\
\e^{- b m} & \text{if} & m \geq C_2 n \, . 
\end{array}
\right.
\end{displaymath}
\end{lemma}

\begin{proof}
First take $0 < r < 1$; let:
\begin{displaymath} 
M (r) = \sup_{|z| = r}  |T_a (z)| = \sup_{|z| = r} \bigg|\frac{z + a}{1 + a z} \bigg|\, \cdot 
\end{displaymath} 
We have $M (r) < 1$, so we can write $M (r) = r^{\rho}$, for some positive $\rho = \rho (a)$. 

The Cauchy inequalities give:
\begin{displaymath} 
|a_{m, n}| \leq \frac{[M (r)]^n}{r^m} = r^{\rho n - m} \, ,
\end{displaymath} 
and we obtain the first inequality  by taking $r = \e^{- \alpha}$ and adjusting $C_1$. 

Next, we use that $T_a$ is analytic on $D(0,1/a)$. Fix $1 < r =: \e^{\beta} < 1 / a$, with $\beta > 0$.  Let $M (r) = \e^{\alpha}$ with $\alpha > 0$. 
The Cauchy inequalities again give:
\begin{displaymath} 
|a_{m, n}| \leq \frac{[M (r)]^n}{r^m} = \e^{\alpha n - \beta m} \, ,
\end{displaymath}
and we obtain the second inequality by adjusting $C_2$. 
\end{proof}

We will also need the following result of V.~\`E~Kacnel'son (\cite{Kacnelson}; see also \cite{Isabelle}, \cite[Corollary 2.2]{IsabelleII}, 
or \cite[Theorem 3.12]{LLQR-comparison}).
\begin{theorem} [V.~\`E~Kacnel'son] \label{theo Kacnelson} 
Let $H$ be a separable complex Hilbert space and $(e_i)_{i \geq 0}$ a fixed orthonormal basis of $H$. 
Let $M \colon H \to H$ be a bounded linear operator. We assume that the matrix of $M$ with respect to this basis is lower-triangular: 
$\langle M e_j \mid e_i \rangle = 0$ for $i < j$.

Let $(\gamma_j)_{j \geq 0}$ be a non-decreasing sequence of positive real numbers and $\Gamma$ the (possibly unbounded) diagonal operator such that 
$\Gamma (e_j) = \gamma_j e_j$, $j \geq 0$. Then the operator $\Gamma^{ - 1} M \, \Gamma \colon H \to H$ is bounded and moreover:
\begin{displaymath} 
\| \Gamma^{- 1} M \, \Gamma \| \leq \| M \| \, .
\end{displaymath}
\end{theorem}
\begin{proof} [Proof of Theorem~\ref{main theorem}]
We may, and do, assume that $\beta$ is non-increasing. 
\smallskip

Proposition~\ref{phi(0)=0} gives the result when $\phi (0) = 0$.
\smallskip

It remains to show that all $C_{T_a}$, $a \in \D$, are bounded. Indeed, if $a = \phi (0)$ and $\psi = T_{- a} \circ \phi$, then $\psi (0) = 0$ and 
$\phi = T_a \circ \psi$, so $C_\phi = C_\psi \circ C_{T_a}$. Moreover, we have only to show that when $a \in [0, 1)$. Indeed, if $a \in \D$ and 
$a = |a|\, \e^{i \theta}$, we have $T_a = R_\theta \circ T_{|a|} \circ R_{- \theta}$, so 
$C_{T_a} = C_{R_{- \theta}} \circ C_{T_{|a|}} \circ C_{R_\theta}$. 
\smallskip

We consider the matrices 
\begin{displaymath} 
A = (a_{m, n} )_{m, n \geq 0} \quad \text{and} \quad A_\beta = \bigg( \sqrt{\frac{\beta_m}{\beta_n}} \, a_{m, n} \bigg)_{m, n \geq 0} \, .
\end{displaymath} 

Since $C_{T_a} e_n = {T_a}^n$, the formula \eqref{T_a power n} shows that $A$ is the matrix of $C_{T_a}$ in $H^2$ with respect to the basis 
$(e_n)_{n \geq 0}$.  On the other hand, $A_\beta$ is the matrix of $C_{T_a}$ in $H^2 (\beta)$ with respect to the basis $(e_n^\beta)_{n \geq 0}$. 
We note that $A_\beta = B A B^{- 1}$, where $B$ is the diagonal matrix with values $\sqrt{\beta_0}, \sqrt{\beta_1}, \ldots$ on the diagonal. 

Since $C_{T_a}$ is a bounded composition operators on $H^2$, the matrix $A$ defines a bounded operator on $\ell_2$. We have to show that 
$A_\beta$ also, i.e. $\| A_\beta \| < \infty$. 
\smallskip

For that purpose, we split $A$ and $A_\beta$ into several sub-matrices. 
\smallskip

Let $N$ be an integer such that $N \geq 2 / C_1$, where $C_1$ is defined in Lemma~\ref{majo coeff} (actually, the proof of that lemma shows that we can 
take $C_1$ such that $1 / C_1$ is an integer, so we could take $N = 2 / C_1$). 
Let $I_0 = [0 , N[\,$ $J_0 = [N, + \infty[\,$ and for $k = 1, 2, \ldots\,$:
\begin{displaymath} 
I_k = [N^k, N^{k  + 1}[\, \quad \text{and} \quad J_k = [N^{k + 1}, + \infty[\,. 
\end{displaymath} 
We define the matrices $D_\beta$ and $R_\beta$, whose entries are respectively:
\begin{align*} 
d_{m, n} & = \left\{ 
\begin{array}{lll}
\dis \sqrt{\frac{\beta_m}{\beta_n}} \, a_{m, n} & \text{if} & \dis (m, n) \in \bigcup_{k = 0}^\infty (I_k \times I_k) \\
\quad 0 & \text{elsewhere;}
\end{array}
\right.
\end{align*}
and
\begin{align*} 
\phantom{+ 1} r_{m, n} & =  \left\{ 
\begin{array}{lll}
\dis \sqrt{\frac{\beta_m}{\beta_n}} \, a_{m, n} & \text{if} & \dis (m, n) \in \bigcup_{k = 0}^\infty (I_k \times I_{k + 1}) \\
\quad 0 & \text{elsewhere.}
\end{array}
\right.
\end{align*} 
We also define the matrix $S_\beta$ with entries:
\begin{displaymath} 
s_{m, n} =  \left\{ 
\begin{array}{lll}
\dis \sqrt{\frac{\beta_m}{\beta_n}} \, a_{m, n} & \text{if} & \dis (m, n) \in \bigcup_{k = 0}^\infty ( J_k \times I_k ) \\
\quad 0 & \text{elsewhere.}
\end{array}
\right.
\end{displaymath} 

Matrices $D$, $R$, and $S$ are constructed in the same way from $A$ and we set $U = A - (D + R + S)$.
\smallskip

Now, let $H_k$ be the subspace of the sequences $(x_n)_{n \geq 0}$ in $\ell_2$ such that $x_n = 0$ for $n \notin I_k$, i.e. 
$H_k = {\rm span}\, \{e_n \tq n \in I_k \}$, and let $P_k$  be (the matrix of) the orthogonal projection of $\ell_2$ with range $H_k$. We have:
\begin{displaymath} 
D = \sum_{k = 0}^\infty P_k A P_k \quad \text{and} \quad R = \sum_{k = 0}^\infty P_k A P_{k + 1} \, ,
\end{displaymath} 
where $D_k = P_k A P_k$ is the matrix with entries $a_{m, n}$ when $(m, n) \in I_k \times I_k$ and $0$ elsewhere, and $R_k = P_k A P_{k + 1}$ the matrix 
with entries $a_{m, n}$ when $(m, n) \in I_k \times I_{k + 1}$ and $0$ elsewhere.

\begin{displaymath} 
\left(
\begin{array} {c@{}|@{}c@{}|@{}c@{}c@{}c@{}c@{}c@{}c}
\begin{array}{|c}
\hline
D_0 \\
\hline
\end{array}
&
\begin{array}{ccc}
\hline
\phantom{R} & R_0 & \phantom{R} \\
\hline
\end{array}
& \phantom{U} & \phantom{U} & \phantom{U} &  \phantom{U} &  \phantom{U} &  \phantom{U} \\
{} & 
\begin{array}{ccc}
\phantom{D} & \phantom{D_1} & \phantom{D} \\
\phantom{D} & D_1 & \phantom{D} \\
\phantom{D} & \phantom{D_1} & \phantom{D} \\
\hline
\end{array}
& 
\begin{array}{ccccc|}
\hline
\phantom{R} & \phantom{R} & \phantom{R_1} & \phantom{R}  & \phantom{R} \\
\phantom{R} & \phantom{R} & R_1 & \phantom{R}  & \phantom{R} \\
\phantom{R} & \phantom{R} & \phantom{R_1} & \phantom{R}  & \phantom{R} \\
\hline
\end{array}
& \phantom{U} & \phantom{U} & \phantom{U}  & \qquad \quad \text{\huge{$U$}} & \\
S_0 & S_1 & 
\begin{array}{ccccc}
\phantom{D} & \phantom{D} & \phantom{D_2} & \phantom{D} & \phantom{D} \\
\phantom{D} & \phantom{D} & \phantom{D_2} & \phantom{D} & \phantom{D} \\
\phantom{D} & \phantom{D} & D_2 & \phantom{D} & \phantom{D} \\
\phantom{D} & \phantom{D} & \phantom{D_2} & \phantom{D} & \phantom{D} \\
\phantom{D} & \phantom{D} & \phantom{D_2} & \phantom{D} & \phantom{D} \\
\hline
\end{array}
&  
\begin{array}{|ccc}
\hline
\phantom{R}  & \phantom{R_2} & \phantom{R} \\
\phantom{R}  & \phantom{R_2} & \phantom{R} \\
\phantom{R} & R_2 & \phantom{R} \\
\phantom{R} & \phantom{R_2} & \phantom{R} \\
\phantom{R}  & \phantom{R_2} & \phantom{R} \\
\hline
\end{array}
& & & \\
& & &
\begin{array}{|ccc} 
 \phantom{R} & \phantom{R_2} & \phantom{R} 
\end{array} 
& & & \\
{} &  {} & S_2 &
\begin{array}{|ccc} 
 \phantom{R} & \phantom{R_2} & \phantom{R} 
\end{array} 
&  &  &  \\
\end{array}
\right)
\end{displaymath} 

\hskip - 1 pt Since the subspaces $H_k$ are orthogonal, the matrices $D$ and $R$ induce bounded operators on $\ell_2$, and
\begin{equation} \label{DR}
\qquad \| D \| \leq \| A \| \, , \quad \| R \| \leq \| A \|  \, . 
\end{equation} 

Now, for $k \geq 1$, let $B_k$ be the diagonal matrix whose entries are $b_{m, m} = \sqrt{\beta_m}$ if $m \in I_k$ and $b_{m, n} = 0$ otherwise. 

Then $P_k D_\beta P_k = P_k B_k D B_k^{- 1} P_k $, so 
\begin{displaymath} 
\| P_k D_\beta P_k \| \leq \| B_k \| \, \| B_k^{- 1} \| \, \| D \| 
\leq \max _{j \in I_k} \sqrt{\beta_j} \, \max_{j \in I_k} \frac{1}{\sqrt{\beta_j}} \,  \| A\| \, .
\end{displaymath} 
But the weight $\beta$ satisfies the $\Delta_2$-condition: $\beta_{2 l} \geq \delta_0 \, \beta_l$, and it follows that for every $l \geq 1$:
\begin{displaymath}
\beta_{N^{2} l} \geq \delta^2 \, \beta_l \, ,
\end{displaymath}
for some other constant $\delta$, chosen small enough to have $\| P_0 D_\beta P_0 \| \leq \delta^{- 1} \| A \|$. 
Since $\beta$ is non-increasing, we have  $\beta_j \geq  \delta^2 \, \beta_{N^k}$ for $N^k \leq j \leq N^{k + 1}$. 
In particular $\max _{j \in I_k} \sqrt{\beta_j} \leq \delta^{- 1} \min_{j \in I_k} \sqrt{\beta_j}$ and 
\begin{displaymath} 
\| P_k D_\beta P_k \| \leq \delta^{- 1} \| A \| \, .
\end{displaymath} 
Hence, by orthogonality of the subspaces $H_k$:
\begin{equation} \label{D_beta}
\| D _\beta \| \leq \delta^{- 1} \| D \| \, .  
\end{equation} 

In the same way, we have $P_k R_\beta P_k = P_k B_k D B_{k + 1}^{- 1} P_k$, so:
\begin{displaymath} 
\| P_k R_\beta P_k \| \leq 
\max _{j \in I_k} \sqrt{\beta_j} \, \max_{j \in I_{k + 1}} \frac{1}{\sqrt{\beta_j}} \,  \| A\| \, 
= \max_{(m, n) \in I_k \times I_{k + 1}} \sqrt{\frac{\beta_m}{\beta_n}} \,  \| A\| \, .
\end{displaymath} 
But, when $(m, n) \in I_k \times I_{k + 1}$, we get 
\begin{displaymath}
\beta_n \geq \beta_{N^{k + 2}} \geq \delta^2 \beta_{N^{k}} \geq \delta^2 \beta_m \, .
\end{displaymath}
Hence 
\begin{equation} \label{R_beta}
\| R_\beta \| \leq \delta^{- 1} \| R \|   \, .
\end{equation} 

Next, consider $U = A - (D + R + S)$; we can compute its Hilbert-Schmidt norm using Lemma~\ref{majo coeff}. 
Note that $u_{m, n} \neq 0$ only (if it happens) when $m \in I_k$ for some $k \geq 0$ and 
$n \geq \min I_{k + 2} \geq N^{k + 2} > N m$, since $m \in I_k$, so only when $m \leq C_1 \, n$. We have, since then $u_{m, n} = a_{m, n}$:
\begin{displaymath} 
\| U \|_{HS}^2 \leq \sum_{n = 0}^\infty \sum_{m\leq C_1 \, n } |a_{m, n} |^2 
\leq \sum_{n = 0}^\infty \sum_{m \leq C_1 n} \e^{- 2 b  n}
\leq \sum_{n = 0}^\infty C_1 n \, \e^{- 2b n} < \infty \, .
\end{displaymath} 

Consequently, with \eqref{DR}, we have $\| S \| \leq \| A \| + \| D \| + \| R \| + \| U \| < \infty$. 
\smallskip

Now, since $S$ is a lower-triangular matrix and $\beta$ is non-increasing, we can use the result of V.~\`E~Kacnel'son (Theorem~\ref{theo Kacnelson}), 
with $\gamma_j = 1 / \sqrt{\beta_j}$. We get that $S_\beta$ defines a bounded operator and $\| S_\beta \| \leq \| S \|$. 
\smallskip

Further, Proposition~\ref{simple proposition polynomial minoration} says that $\beta$ has a polynomial minoration: 
\begin{displaymath} 
\beta_n \geq c \, n^{- \sigma} 
\end{displaymath} 
for positive constants $c$ and $\sigma$. Then, if $U_\beta = A_\beta - (D_\beta + R_\beta + S_\beta)$, we have:
\begin{displaymath} 
\| U_\beta \|_{HS}^2 \leq \sum_{n = 0}^\infty \sum_{m < \rho n / 2} \frac{|a_{m, n}|^2}{\beta_n} 
\leq \sum_{n = 0}^\infty \frac{\rho n}{2} \, \e^{- \alpha \rho n} \frac{\, n^\sigma}{c} < \infty \, .
\end{displaymath} 

Putting this together with \eqref{D_beta} and \eqref{R_beta}, we finally obtain that $A_\beta = S_\beta + D_\beta + R_\beta + U_\beta$ is 
the matrix of a bounded operator, and that ends the proof of Theorem~\ref{main theorem}.
\end{proof}
\smallskip

\noindent {\bf Remark.} We could have done here without the theorem of Kacnel'son, using Lemma~\ref{majo coeff} to show that the matrix 
$\Big(a_{m, n} \sqrt{\frac{\beta_m}{\beta_n}}\Big)$ is Hilbert-Schmidt as well ``far below'' the main diagonal. Indeed, we see from this lemma that
\begin{displaymath} 
\sum_{m \geq C_{2} n} \!\! |a_{m, n}|^2 \frac{\beta_m}{\beta_n}
 \lesssim \sum_{m \geq C_{2} n} \!\! n^{\sigma} |a_{m,n}|^2 
\lesssim \sum_{m \geq C_{2} n} \!\! n^{\sigma} \e^{- 2 b m} \lesssim \sum_{n \geq 0} n^{\sigma} \e^{- 2 b C_{2} n} < + \infty \, .
\end{displaymath}
See also Theorem~\ref{main bis} of the final section. Kacnel'son's theorem will really be needed in the forthcoming Theorem~\ref{theo cond suff}. 

\goodbreak

\subsection {Sufficient condition for the boundedness of composition operators -- Part II} \label{sec: theo cond suff} 

In this section, we give a sufficient condition for the boundedness of composition operators with symbol vanishing at $0$, of a different nature than 
the one given in Proposition~\ref{phi(0)=0}.

\begin{theorem} \label{theo cond suff}
Let $\beta = (\beta_n)_{n = 0}^\infty$ be a bounded sequence of positive numbers with a polynomial minoration. Assume that:
\begin{equation} \label{cond Luis}
\begin{split}
\text{For every } \delta > 0, & \text{ there exists } \text{a positive constant } C = C (\delta) \text{ such that } \\
& \hskip - 0.8 em \beta_m  \leq C \, \beta_n \quad \text{whenever } m > (1 + \delta) \, n \, .
\end{split}
\end{equation}

Then,  for all symbols $\phi \colon \D \to \D$ vanishing at $0$, the composition operator $C_\phi$ is bounded on $H^2 (\beta)$.
\end{theorem}

To prove Theorem~\ref{theo cond suff}, we need several lemmas.

\begin{lemma} \label{lemma 1}
Let $\phi \colon \D \to \D$ be an analytic self-map such that $\phi (0) = 0$ and $| \phi ' (0) | < 1$. Then there exists $\rho > 0$ such that
\begin{displaymath}
|\hat{\phi^n} (m) | \leq \exp \Big( - \frac{1}{2} \, [(1 + \rho) \, n - m] \Big) \, .
\end{displaymath}
\end{lemma}
\begin{proof}
It is the same as that of Lemma~\ref{majo coeff}. Since $\phi (0) = 0$, we can write $\phi (z) = z\, \phi_1 (z)$. Since $| \phi ' (0) | < 1$, we 
have $\phi_1 \colon \D \to \D$. Let $M (r) = \sup_{|z| = r} | \phi_1 (z) |$. The Cauchy inequalities say that 
$|\hat {\phi_1^n} (m)| \leq [M (r)]^n / r^m$. We have $M (r) < 1$, so there exists a positive number $\rho = \rho (r)$ such that 
$M (r) = r^\rho$. We get:
\begin{displaymath}
|\hat {\phi^n} (m)| = |\hat {\phi_1^n} (m - n)| \leq \frac{r^{\rho n}}{r^{m - n}} = r^{(1 + \rho)\, n - m} \, ,
\end{displaymath}
and the result follows, by taking $r = \e^{- 1 / 2}$.
\end{proof}

The next lemma is a variant of the result of V.~\`E~Kacnel'son quoted before. 
\begin{lemma} \label{lemma 2}
Let $A \colon \ell_2 \to \ell_2$ be a bounded operator represented by the matrix $\big( a_{m, n} \big)_{m, n}$, i.e. 
$a_{m, n} = \langle A \, e_n, e_m \rangle$, where $(e_n)_{n \geq 1}$ is the canonical basis of $\ell_2$. 

Let $(d_n)$ be a sequence of positive numbers such that, for every $m$ and $n$:
\begin{equation} \label{cond trou}
d_m < d_n \quad \Longrightarrow \quad a_{m, n} = 0 \, .
\end{equation}

Then, $D$ being the (possibly unbounded) diagonal operator with entries $d_n$, we have:
\begin{displaymath}
\| D^{- 1} A D \| \leq \| A \| \, .
\end{displaymath}
\end{lemma}

For the convenience of the reader, we reproduce the proof.
\begin{proof}
Let $\C_0$ be the right-half plane $\C_0 = \{z \in \C \tq \Re z > 0 \}$. We set $H_N = {\rm span}\, \{ e_n \tq n \leq N \}$ and
\begin{displaymath}
A_N = P_N A J_N \, ,
\end{displaymath}
where $P_N$ is the orthogonal projection from $\ell_2$ onto $H_N$ and $J_N$ the canonical injection from $H_N$ into $\ell_2$. 
We consider, for $z \in \overbar{\C_0}$:
\begin{displaymath}
A_N (z) = D^{- z} A_N D^z \colon H_N \to H_N \, ,
\end{displaymath}
where $D^z (e_n) = d_n^{\, z} e_n$. 

If $\big( a_{m, n} (z) \big)_{m, n}$ is the matrix of $A_N (z)$ on the basis $\{e_n \tq n \leq N\}$ of $H_N$, we clearly have:
\begin{displaymath}
a_{m, n} (z) = a_{m, n} (d_n / d_m)^z \, .
\end{displaymath}
In particular, we have, thanks to \eqref{cond trou}:
\begin{displaymath}
a_{m, n} (z) = 0  \quad \text{if } d_m < d_n \, ,
\end{displaymath}
and
\begin{displaymath}
\qquad\qquad\quad | a_{m, n} (z) | \leq \sup_{k, l} |a_{k, l}| := M \, , \qquad \text{for all } z \in \overbar{\C_0} \, .
\end{displaymath}
Since $\| A_N (z) \|^2 \leq \| A_N (z)\|_{HS}^2 = \sum_{m, n \leq N} |a_{m, n} (z)|^2 \leq (N + 1)^2 M^2$, we get:
\begin{displaymath}
\qquad \qquad \| A_N (z) \| \leq (N + 1) \, M \qquad \text{for all } z \in \overbar{\C_0} \, .
\end{displaymath}

Let us consider the function $u \colon \overbar{\C_0} \to \overbar{\C_0}$ defined by:
\begin{equation}
u_N (z) = \| A_N (z) \| \, .
\end{equation}
This function $u_N$ is continuous on $\overbar{\C_0}$, bounded above by $(N + 1) M$, and subharmonic in $\C_0$. Moreover, thanks to 
\eqref{cond trou}, the maximum principle gives:
\begin{displaymath}
\sup_{\overbar{\C_0}} u_N (z) = \sup_{\partial \C_0} u_N (z) \, .
\end{displaymath}
Since $\| D^z \| = \| D^{- z} \| = 1$ for $z \in \partial \C_0$, we have $\| A_N (z) \| \leq \| A_N \|$ for $z \in \partial \C_0$, and we get:
\begin{displaymath}
\sup_{\overbar{\C_0}} u_N (z) \leq \| A_N \| \leq \| A \| \, .
\end{displaymath}
In particular $u_N (1) \leq \| A \|$, and, letting $N$ going to infinity, we get $\| D^{- 1} A D \| \leq \| A \|$.
\end{proof}
\begin{proof} [Proof of Theorem~\ref{theo cond suff}]
First, if $|\phi ' (0) | = 1$, we have $\phi (z) = \alpha \, z$ for some $\alpha$ with $|\alpha | = 1$, and the result is trivial.

So, we assume that $|\phi ' (0) | < 1$. Then, by Lemma~\ref{lemma 1}, there exists $\rho > 0$ such that, for all $m$, $n$:
\begin{displaymath}
| \hat{\phi^n} (m) | \leq \exp \Big( - \frac{1}{2} \, [(1 + \rho) \, n - m] \Big) \, .
\end{displaymath}

Since $\phi (0) = 0$, we also know that $\hat{\phi^n} (m) = 0$ if $m < n$. 
\smallskip

Take $\delta = \rho / 2$ and use property \eqref{cond Luis}: there exists $C > 0$ such that:
\begin{displaymath}
\qquad \frac{\beta_m}{\beta_n} \leq C \quad \text{when } m \geq (1 + \delta) \, n \, .
\end{displaymath}

Define now a new sequence $\gamma = (\gamma_n)$ as:
\begin{displaymath}
\gamma_n = \max \bigg\{ \beta_n, \sup_{m > (1 + \delta) \, n} \beta_m \bigg\} \, .
\end{displaymath}

We have:
\smallskip

1) $\beta_n \leq \gamma_n \leq C \, \beta_n$;
\smallskip

2) $\gamma_m \leq \gamma_n$ \quad if $m \geq (1 + \delta) \, n$.
\medskip

Item 1) implies that $H^2 (\gamma) = H^2 (\beta)$, and we are reduced to prove that 
$C_\phi \colon H^2 (\gamma) \to H^2 (\gamma)$ is bounded.
\smallskip

Let $A = \big( a_{m, n} \big) = \big( \hat{\phi^n} (m) \big)$. We have to prove that
\begin{displaymath}
B = \big( \gamma_m^{1 / 2} \gamma_n^{- 1 / 2} a_{m, n} \big)_{m, n}
\end{displaymath}
represents a bounded operator on $\ell_2$.
\smallskip

Define the matrix
\begin{displaymath}
A_1 = \big( a_{m, n} \ind_{\{ (m, n) \tq m \leq (1 + \delta) \, n \} } \big)_{m, n}
\end{displaymath}
and set $A_2 = A - A_1$. Define analogously $B_1$ and $B_2 = B - B_1$.

Then $A_1$ is a Hilbert-Schmidt operator, because (recall that $a_{m, n} = 0$ if $m < n$)
\begin{align*}
\sum_{n = 1}^\infty \sum_{m = 1}^{(1 + \delta) \, n} |a_{m, n} |^2 
& \leq \sum_{n = 1}^\infty \sum_{m = n}^{(1 + \delta) \, n} \exp \big( - [(1 + \rho) \, n - m] \big) \\
& \leq \sum_{n = 1}^\infty (\delta \, n + 1) \, \exp ( - \delta \, n ) < \infty \, .
\end{align*}

Now, $\beta$ is bounded above and has a polynomial minoration, so, for some positive constants $C_1$, $C_2$, and $\alpha$, we have:
\begin{align*}
\sum_{n = 1}^\infty \sum_{m = n}^{(1 + \delta) \, n} \frac{\gamma_m}{\gamma_n} \, |a_{m, n} |^2 
& \leq \sum_{n = 1}^\infty \sum_{m = n}^{(1 + \delta) \, n} \frac{C_1 \ }{n^{- \alpha}} \, \exp ( - \delta \, n) \\
& \leq \sum_{n = 1}^\infty C_2 \, n^{\alpha  + 1} \exp ( - \delta \, n) < \infty \, ,
\end{align*}
meaning that $B_1$ is a Hilbert-Schmidt operator.
\smallskip

Since $A$ is bounded, it follows that $A_2 = A - A_1$ is bounded. Remark that, writing $A_2 = \big( \alpha_{m, n} \big)_{m, n}$, we 
have, with $d_n = 1 / \sqrt{\gamma_n}$:
\begin{displaymath}
d_m < d_n \quad \Longrightarrow \quad \gamma_m > \gamma_n \quad \Longrightarrow \quad m < (1 + \delta) \, n \quad \Longrightarrow 
\quad \alpha_{m, n} = 0 \, .
\end{displaymath}
Hence we can apply Lemma~\ref{lemma 2} to the matrix $A_2$, and it ensues that $B_2$ is bounded, and therefore that $B = B_1 + B_2$ is 
bounded as well, as wanted.
\end{proof}
\medskip

As a corollary of Theorem~\ref{theo cond suff}, we can provide the following example.
\begin{theorem} \label{theo example}
There exists a bounded sequence $\beta$, with a polynomial minoration, but which is \emph{not essentially decreasing}, for which every composition 
operator with symbol vanishing at $0$ is bounded on $H^2 (\beta)$, with nevertheless $\sup_{\phi (0) = 0} \| C_\phi \| = \infty$.
\end{theorem}

It should be noted that for this weight, the composition operators are not all bounded, as we will see in Proposition~\ref{T_a not bounded}.

\begin{proof} 
Define $\beta_n = 1$ for $n \leq 3!$, and, for $k \geq 3$:
\begin{displaymath}
\left\{
\begin{array} {lcl} 
\beta_n = & \dis \frac{1}{k!} & \text{for } k! < n \leq (k + 1)! - 2 \text{ and for } n = (k + 1)! \\
\\
\beta_n = & \dis \frac{1}{(k + 1)!} & \text{for } n = (k + 1)! - 1 \, .
\end{array}
\right.
\end{displaymath}

Note that, for $m > n$, we have $\beta_m > \beta_n$ only if $n = (k + 1)! - 1$ and $m = (k + 1)! = n + 1$, for some $k \geq 3$. 
\smallskip

However, $\beta$ is not essentially decreasing since, for every $k \geq 3$, we have $\beta_{n + 1} / \beta_n = k + 1$ if $n = (k + 1)! - 1$. 
\smallskip

The sequence $\beta$ has a polynomial minoration because $\beta_n \geq 1 / (2 \, n)$ for all $n \geq 1$. In fact, for $k \geq 3$, we have 
$\beta_n \geq (k + 1)/ n \geq 1 / n$ if $k! < n \leq (k + 1)! - 2$ or if $n = (k + 1)!$; and for $n = n = (k + 1)! - 1$, we have 
$n \, \beta_n = [(k + 1)! - 1] / (k + 1)! \geq 1 /2$. 
\smallskip

Now, it remains to check \eqref{cond Luis} in order to apply Theorem~\ref{theo cond suff} and finish the proof of Theorem~\ref{theo example}. 
Note first that we have $\beta_m / \beta_n \leq 1$ if $m \geq n + 2$. Next, for given $\delta > 0$, there exists an integer $N$ such that 
$(1 + \delta) \, n \geq n + 2$ for every $n \geq N$, so $\beta_m / \beta_n \leq 1$ if $m \geq (1 + \delta)\, n$ and $n \geq N$. 
It suffices to take $C = \max_{1 \leq n \leq N} \beta_{n + 1} / \beta_n$ to obtain \eqref{cond Luis}. The last assertion follows from Proposition~\ref{phi(0)=0}.
\end{proof}
\goodbreak

\section {Necessity of the $\Delta_2$-condition} \label{sec: necessity} 

In this section, we will show that, for essentially decreasing sequences $\beta$, the $\Delta_2$-condition is necessary for having boundedness of 
composition operators on $H^2 (\beta)$. We will indeed show slightly more.
\goodbreak

\begin{theorem} \label{theo necessity Delta-deux}
Let $\beta$ be such that, for some $a \in (0, 1)$, $T_a$ induces a bounded composition operator on 
$H^2 (\beta)$. Then $\beta$ satisfies: 
\begin{align*}
(\exists\,  0 < \delta < & \, 1/3) \ (\forall n \geq 1) \ (\exists \, E_n \subseteq [(1-2\delta)n, (1 - \delta) n] ) \quad  \text{with} \\ 
& |E_n| \geq \delta n \quad \text{and} \quad \beta_n \geq \delta \frac{1}{|E_n|} \sum_{m \in E_n} \beta_m \, . 
\end{align*}
In particular, if $\beta$ is essentially decreasing, then $\beta$ satisfies the $\Delta_2$-condition.
\end{theorem}

In order to prove Theorem~\ref{theo necessity Delta-deux}, we need several preliminary lemmas. The first one is standard, but we give it for convenience.

\begin{lemma} \label{lemma Poisson} 
Let $a \in (0, 1)$ and let 
\begin{displaymath} 
P_{- a} (x) = \frac{1 - a^2}{1 + 2 \, a \cos x + a^2} 
\end{displaymath} 
be the Poisson kernel at the point $- a$. Then, for all $x \in [- \pi, \pi ]$:
\begin{equation} 
T_{a} (\e^{i x}) = \exp \big[ i \, h_a (x) \big] \, ,
\end{equation} 
where 
\begin{equation} 
h_{a} (x) = \int_{0}^x P_{- a} (t) \, dt \, . 
\end{equation} 
\end{lemma}
\begin{proof}
For $t \in [- \pi, \pi]$,  write:
\begin{displaymath} 
\psi (t) := \frac{\e^{it} + a}{1 + a\, \e^{it}} = \exp(i \, u (t) ) \, ,
\end{displaymath} 
with $u$ a real-valued, ${\cal C}^1$ function on $[- \pi, \pi]$ such that $u (0) = 0$. This is possible since $|\psi (\e^{it})| = 1$ and $\psi (0) = 1$. 
Differentiating both sides with respect to $t$, we get:
\begin{displaymath} 
i \, \e^{it} \frac{1 - a^2}{(1 + a \, \e^{it})^2} = i \, u' (t) \, \frac{\e^{it} + a}{1 + a \, \e^{it}} \, \cdot
\end{displaymath} 
This implies 
\begin{displaymath} 
u' (t) = \frac{1 - a^2}{|1 + a \, \e^{it}|^2} = P_{- a} (t) \, ,
\end{displaymath} 
and the result follows since $u (0) = 0 = h_a (0)$. 
\end{proof}
%

\subsection {Main lemma and proof of Theorem~\ref{theo necessity Delta-deux}} \label{subsec: main lemma}

To prove Theorem~\ref{theo necessity Delta-deux}, we need the following lemma. 
The proof of this lemma uses Lemma~\ref{lemma Poisson} and a van der Corput type estimate, inspired from \cite[pp.~72--73]{TIT}. 
We thank R.~Zarouf \cite{ZAR} for interesting recent informations in this respect, related to his joint work with O.~Szehr on the Sch\" affer problem 
(see \cite{SZZA}, in which the authors are primarily concerned with upper bounds). 

The first version of our paper was put on arXiv at the end of November 2020. Since then, the paper \cite{Bo-Fo-Za} was put on arXiv on July 2021, 
where sharp estimates of powers of Blaschke factors are given (see also K.~Fouchet's thesis \cite{Fouchet}), with different purposes (strongly annular 
analytic functions). However, in our case, our proof is much simpler.
\smallskip

Recall that we have set:
\begin{equation} 
[T_{a} (z)]^{n} = \sum_{m = 0}^\infty a_{m, n} z^m \, .
\end{equation} 
\goodbreak

\begin{lemma} \label{main lemma}
Let $a \in (0, 1)$. We set:
\begin{equation} \label{def tau}
\tau = \frac{1 + a}{1 - a} > 1 
\end{equation}
and write:
\begin{equation} \label{def mu}
\tau^{- 1} = 1 - 3 \mu \, ,
\end{equation} 
with $\mu = \mu_a \in (0, 1/3)$. For every fixed positive integer $n$, let:
\begin{equation} 
J_n = [(1 - 2 \mu) \, n, (1 - \mu) \, n] \, .
\end{equation} 
Then, there exists $\delta = \delta_a > 0$ such that, for every $n$ large enough, there exists a set of indices $E_n \subseteq J_n$ with cardinality 
$|E_n| \geq \delta n$ and such that:
\begin{equation} \label{suth}   
m \in E_n \quad \Longrightarrow  \quad |a_{m, n}|\geq \delta \, n^{- 1/2} \, . 
\end{equation} 
\end{lemma}
\begin{proof} [Proof of Theorem~\ref{theo necessity Delta-deux}] 
\smallskip

Set $M = \| C_{T_a} \|$. We have:
\begin{equation} \label{T_a bounded}
\sum_{m = 0}^\infty |a_{m, n}|^2 \beta_m = \| {T_a}^n \|^2 = \| C_{T_a} (z^n) \|^2 \leq \| C_{T_a} \|^2 \| z^n \|^2 = M^2 \beta_n \, . 
\end{equation} 
so, by Lemma~\ref{main lemma}, since $E_n \subseteq J_n = [(1 - 2 \mu) n, (1 - \mu) n]$:
\begin{displaymath} 
M^2 \beta_n \geq \sum_{m \in E_n} |a_{m, n}|^2 \beta_m \geq \delta^2 \, n^{- 1} \sum_{m \in E_n} \beta_m  
\geq \delta^3 \, |E_n|^{-1} \,  \sum_{m \in E_n} \beta_m \, . 
\end{displaymath} 

This proves (changing $\delta$) the first part of Theorem \ref{theo necessity Delta-deux}. Next, assume that $\beta$ is essentially decreasing. We may, 
and do, assume that $\beta$ is non-increasing. We set, for $x$ not an integer, $\beta_x = \beta_k$ with $k$ the least integer greater than $x$. 
The above implies, for all integers $n \geq 1$: 
\begin{displaymath} 
\beta_n \geq (\delta^3 / M^2) \, |E_n|^{- 1} |E_n| \, \beta_{(1 - \mu) n}\geq c \, \beta_{(1 - \mu) n} \, .
\end{displaymath}

Let $r \geq 1$ such that $(1 - \mu)^r \leq 1 /2$; we have:
\begin{displaymath} 
\beta_n \geq c^r \, \beta_{(1 - \mu)^r n} \geq c^r \, \beta_{n / 2} \, ,
\end{displaymath} 
so $\beta$ satisfies the $\Delta_2$-condition. 
\end{proof}

A consequence of Theorem~\ref{theo necessity Delta-deux} is the following result.
\begin{proposition} \label{T_a not bounded}
For the weight $\beta$ constructed in the proof of Theorem~\ref{theo example}, no automorphism $T_a$ with $0 < a < 1$ can be bounded.
\end{proposition}
\begin{proof}
Consider the necessary condition for the boundedness of $C_{T_a}$ in Theorem~\ref{theo necessity Delta-deux}:
\begin{equation} \label{necnec}  
\beta_n \geq \delta \, \frac{1}{|E_n|} \sum_{m \in E_n} \beta_m \, .
\end{equation}
For the weight $\beta$ constructed in the proof of Theorem~\ref{theo example}, we are going to see that this condition \eqref{necnec} is not satisfied 
for $n = (k + 1)! - 1 =: n_k$. 

Indeed, for this $n$, the left-hand side of \eqref{necnec} is equal to $1 / (k + 1)!$ and the right-hand side to $\delta / k!$ since $(1 - 2 \delta) n_k > k!$ 
for $k$ large, so that all $\beta_m$ are equal to $1 / k!$ for $m \in E_{n_k}$. This ends the proof.
\end{proof}

%
\subsection {Proof of Lemma~\ref{main lemma}}

To prove Lemma~\ref{main lemma}, we will use a variant of \cite[Lemma~4.6 p.~72]{TIT} on the stationary phase method. 
A careful reading of the proof in \cite[ p.~72]{TIT} gives the version below, which allows the derivative $F '$ of $F$ to vanish at some point, as occurs in 
our situation. For sake of completeness, we will give a proof, however postponed.
\goodbreak

\begin{proposition} [Stationary phase] \label{tihebr}  
Let $F$ be real function on the interval $[ A, B]$, with continuous derivatives up to the third order and $F'' > 0$ throughout $]A, B[$.  Assume that 
there is a (unique) point $c$ in $]A, B[$ such that $F' (c) = 0$, and that, for some positive numbers $\lambda_2$, $\lambda_3$, and $\eta$, the following 
assertions hold: 
\begin{enumerate}
\setlength\itemsep {-0.1 em}

\item [$1)$] $[c - \eta, c + \eta] \subseteq [A, B]$;

\item [$2)$] $F '' (x) \geq \lambda_2$ for all $x \in [c - \eta, c + \eta]$;

\item [$3)$]  $|F ''' (x)| \leq  \lambda_3$  for all $x \in [A, B]$.
 \end{enumerate}
Then:
\begin{equation} \label{ozouf} 
\int_A^B \e^{i F(x)} \, dx = \sqrt{2 \pi} \, \, \frac{\, \e^{i (F (c) + \pi/4)} \, }{\ |F '' (c)|^{1/2}} + O \, \bigg( \frac{1}{\eta \lambda_2} + \eta^4 \lambda_3 \bigg) \, ,
\end{equation}
where the $O$ involves an absolute constant. 
\end{proposition}
\goodbreak

\begin{proof} [Proof of Lemma~\ref{main lemma}]
We turn to the problem of bounding $a_{m, n}$ from below, \emph{in the case $m \in J_n$, and only in that case}. 
Since $\inf_{[0, \pi]} P_{- a} = \tau^{- 1}< \sup_{[0, \pi]}P_{- a} = \tau$, there exists a unique point $x_m = x_{m, n} \in [0, \pi]$ such that
\begin{displaymath} 
n P_{- a} (x_m) - m = n \, \frac{(1 - a^2)}{1 + 2 \, a \cos x_m + a^2} - m = 0 \, ,
\end{displaymath} 
or else:
\begin{equation} \label{cosinus}
\cos x_m = \frac{n}{m} \frac{1 - a^2}{2 \, a} - \frac{1 + a^2}{2 \, a} \, \cdot 
\end{equation} 

The point is that if $m \in J_n$, $x_m$ can approach neither $0$ nor $\pi$, so that $\sin x_m \geq \delta_a > 0$; 
more precisely, the definition of $J_n$ and \eqref{cosinus} imply that $\pi / 4 \leq x_m \leq \pi / 2$. 

With $h_a$ the function of Lemma~\ref{lemma Poisson}, the Fourier formulas give, since $a_{m, n}$ is real, or since $h_a (x) - m x$ is odd: 
\begin{displaymath} 
2 \pi a_{m, n} = \int_{- \pi}^{\pi} \exp i [n h_a (x) - m x] \, dx = 2 \, \Re I_{m, n} \, , 
\end{displaymath} 
where 
\begin{equation} 
I_{m, n} = \int_{0}^{\pi} \exp i [n h_a (x) - m x ] \, dx \, .
\end{equation} 
Write:
\begin{equation} 
I_{m, n} = \int_{0}^{\pi} \exp [i \, F_m (x) ] \, dx \, ,
\end{equation} 
with:
\begin{equation} \label{phase} 
\qquad F_m (x) = n h_a (x) - m x = n \int_{0}^x P_{- a} (t) \, dt - m x \, .
\end{equation} 
We have:
\begin{equation}  \label{derivative phase} 
F_m ' (x) = n \, P_{- a} (x) - m \, .
\end{equation} 

We will now proceed in two steps, first giving good lower bounds for $|I_{m, n}|$, then showing that the argument of $I_{m, n}$ is often far from 
$\pi/2$ mod.~$\pi$. Then, we will be done. 
\medskip

\noindent {\bf First step.} We will prove that:
\begin{equation} \label{cruc} 
I_{m, n} = \sqrt{2 \pi} \, n^{- 1/2} \,\, \frac{\e^{i \, (F_m (x_m) + \pi/4)}}{\sqrt{|h_a '' (x_m)|}} + O \, (n^{- 3/5}) \, ,
\end{equation} 
where the $O$ only depends on $a$. 
\smallskip

Note that  $3/5 > 1/2$  and  $F_m '' = n \, h_a ''$.
\smallskip

To get \eqref{cruc}, we will show that Theorem~\ref{tihebr} is applicable with:
\begin{displaymath} 
[A, B] = [0, \pi] \, , \quad  c = x_m \, , \quad \lambda_2 = \kappa_0 \, n \, , \quad  \lambda_3 = C_0 n \, , \quad  
\eta = (\lambda_{2} \lambda_{3})^{- 1/5} \, .
\end{displaymath} 
The parameter $\eta$ is chosen in order to make both error terms in Theorem~\ref{tihebr} equal:  $ \frac{1}{\eta \lambda_2} = \eta^4 \lambda_3$; so:
\begin{displaymath} 
\eta = \kappa \, n^{- 2/5}   
\end{displaymath} 
and 
\begin{equation} \label{big O}
\frac{1}{\eta \lambda_2} + \eta^4 \lambda_3 = \tilde \kappa \, n^{- 3/5} = O\, (n^{- 3/5}) 
\end{equation} 
(with $\kappa = (\kappa_0 C_0)^{- 1 / 5}$ and $\tilde \kappa = 2 / \kappa_0 \kappa$). 
\smallskip

The slight technical difficulty encountered here is that $F_m ''(x)$ vanishes at $0$ and $\pi$. But Theorem~\ref{tihebr} covers this case. We have 
\begin{displaymath} 
F_m ''(x) = n P'_{- a} (x) = 2 \, a (1 - a^2) \, \frac{\sin x}{(1 + 2 a \cos x + a^2)^2} \, n \, ,
\end{displaymath} 
and there are some positive $\kappa_0$ and $\sigma$ such that 
\begin{equation} \label{rectif} 
F_m ''(x) \geq \kappa_0 \, n = \lambda_2  \quad \text{for } x \in [\sigma , \pi - \sigma] \, .
\end{equation}

Now (for $n$ large enough), $[x_m - \eta, x_m + \eta] \subseteq  [\sigma, \pi - \sigma]$. Hence the assumptions $1)$ and $2)$ of 
Proposition~\ref{tihebr} are satisfied. 
\smallskip

Finally, since $F_m (x) = n h_a (x) - m x$, and $h_a$ is ${\cal C}^\infty$ on $\R$, we have, for all $x \in [0, \pi]$:
\begin{displaymath} 
|F_m ''' (x) | \leq  C_0 \, n = \lambda_3 \, ,
\end{displaymath} 
and assertion $3)$ of Proposition~\ref{tihebr} holds. 
\smallskip

With \eqref{big O} this ends the proof of \eqref{cruc}, once we remarked that $n h_a '' (x_m) = F_m '' (x_m)$. 
\smallskip

Note that, since  $|h_a '' (x_m)|\leq M_a$, we get that $|I_{m, n}| \geq \delta \, n^{- 1/2}$ when  $m \in J_n$.
\medskip

\noindent {\bf Second step.} The mean-value theorem gives, for $m \in J_n$:
\begin{equation} \label{mvth} 
\quad\!\! |\sin x_m| \geq \delta \quad \text{and} \quad   x_{m + 1} - x_m \approx \cos x_m - \cos x_{m + 1} \, . 
\end{equation}
We also have, for $x \in {\cal J} = [1 - 2 \mu, 1 - \mu]$, with another constant $\delta$: 
\begin{equation} \label{alha} 
\delta \, \leq  P'_{- a} (x) = 2 \, a (1 - a^2) \, \frac{\sin x}{(1 + 2 a \cos x + a^2)^2} \leq \delta^{- 1} \, . 
\end{equation} 

We now claim that
\begin{equation} \label{clth} 
x_{m + 1} - x_m \approx n^{- 1} \quad \text{for } m \in J_n \, . 
\end{equation} 
Indeed, since $m \in J_n$, we have, using \eqref{cosinus}: 
\begin{displaymath} 
\cos x_m - \cos x_{m + 1} = \frac{1 - a^2}{2 a} \, \frac{n}{m (m + 1)} \approx \frac{n}{m^2} \approx n^{- 1} \, .
\end{displaymath} 
In view of \eqref{mvth}, this proves \eqref{clth}. 

Now, according to \eqref{cruc}, when $m \in J$, the main term in $I_{m, n}$ is
\begin{displaymath} 
A_{m, n} := n^{- 1/2}\,\frac{\sqrt{2 \pi}}{\sqrt{|h_a '' (x_m)|}} \, \e^{i (F_m (x_m) + \pi/4)} \, ,
\end{displaymath} 
and its argument $\theta_m$ is $F_m(x_m) + \pi/4$.  Going from $m$ to $m + 1$, the variation $F_{m + 1} (x_{m + 1}) - F_m (x_m) $ of this argument is
\begin{align*} 
\theta_{m + 1} - \theta_m 
& = n \int_{x_m}^{x_{m+  1}} \bigg( P_{- a} (t) - \frac{m}{n} \bigg) \, dt - x_{m + 1} \\ 
& = n \int_{x_m}^{x_{m + 1}} \bigg( P_{- a} (t) - P_{- a} (x_m) \bigg) \, dt - x_{m + 1} \, .
\end{align*} 
But, due to \eqref{alha}, this implies:
\begin{align*} 
\theta_{m + 1} - \theta_m 
& = - x_{m + 1} + {\rm O}\, \bigg( n \int_{x_m}^{x_{m + 1}} (t - x_m) \, dt \bigg) \\ 
& = - x_{m + 1} + {\rm O} \, \big( n \, (x_{m + 1} - x_m)^2 \big) = - x_{m + 1} + {\rm O}\, (n^{- 1}) \, .
\end{align*} 
Since $- \sigma \leq - x_{m + 1} \leq \pi - \sigma$, the variation of  $\theta_m$ is thus regular. 
As a 
consequence, for a positive proportion $E_n$ of the indices $m \in J_n$, the argument $\theta_m$ will belong to a subarc of $\T$ which lies $\delta$-apart from 
$\pm \pi/2$, implying $\cos \theta_m \geq \delta$, or else:  
\begin{displaymath} 
|E_n| \geq \delta \, n \, ,
\end{displaymath} 
and, for all $m \in E_n$:
\begin{displaymath} 
\Re A_{m, n} \geq \delta \, |A_{m,n}| \geq \delta \, n^{- 1/2}.  
\end{displaymath} 
It follows that, for $m \in  E_n$: 
\begin{displaymath} 
\Re  I_{m, n} \geq \delta \, n^{- 1/2} - C \, n^{- 3/5} \geq \tilde\delta \, n^{- 1/2} \, .
\end{displaymath} 
Since $a_{m, n} = \pi^{- 1} \Re  I_{m, n}$, that ends the proof of Lemma~\ref{main lemma}. 
\end{proof}
\goodbreak


We now pass to the proof of Proposition~\ref{tihebr}. The following lemma can be found in \cite[Lemma~1, page 47]{Montgomery}. 
\goodbreak

\begin{lemma} \label{spc} 
Let $F \colon [u,v] \to \R$, with $u < v$, be a ${\cal C}^2$- function with $F'' > 0$, and $F '$ not vanishing on $[u, v]$. Let 
\begin{displaymath}
J = \int_{u}^v e^{i F(x)} \, dx \, . 
\end{displaymath}
Then: 
\begin{itemize}
\setlength\itemsep {-0.1 em}

\item [{\rm a)}] if $F ' > 0$ on $[u, v]$, then $|J| \leq \frac{2}{ F ' (u)}\,$;

\item [{\rm b)}] If $F ' < 0$ on $[u, v]$, then $|J| \leq \frac{2}{|F ' (v)|} \, \cdot$
\end{itemize}
\end{lemma}
\goodbreak

\begin{proof} [Proof of Proposition~\ref{tihebr}] 

Write now the integral $I$ of Proposition~\ref{tihebr} on $[A, B]$ as $I = I_1 + I_2 + I_3$ with:
\begin{displaymath}
I_1 = \int_A^{c - \eta} \e^{i F(x)} \, dx \, , \quad  I_2 = \int_{c - \eta}^{c + \eta} \e^{i F(x)} \, dx \, , 
\quad  I_3 = \int_{c + \eta}^B \e^{i F(x)} \, dx \, .
\end{displaymath}
Lemma~\ref{spc} with $u = A$ and $v = c - \eta$ implies:
\begin{equation} \label{ein} 
|I_1| \leq  \frac{2}{|F ' (c -  \eta)|} \leq \frac{2}{\eta \lambda_2} \, \raise 1 pt \hbox{,}  
\end{equation}
where, for the last inequality, we just have to write 
\begin{displaymath}
|F ' (c - \eta)| = F ' (c) - F ' (c - \eta) = \eta \, F '' (\xi) 
\end{displaymath}
for some $\xi \in [c - \eta, c]$ so that $F '' (\xi) \geq \lambda_2$. 
\smallskip

Similarly, Lemma~\ref{spc} with $u = c + \eta$ and $v = B$ implies 
\begin{equation} \label{zwei} 
|I_3| \leq  \frac{2}{F ' (c + \eta)} \leq \frac{2}{\eta \lambda_2} \, \cdot 
\end{equation}

We can now estimate $I_2$. The Taylor formula shows that
\begin{displaymath}
F (x) = F (c) + \frac{(x - c)^2}{2} F '' (c) + R \, ,
\end{displaymath}
with
\begin{displaymath}
|R| \leq \frac{|x - c|^3}{6} \, \lambda_3 \, .
\end{displaymath}
Hence 
\begin{displaymath}
I_2 = \e^{i F(c)} \int_{0}^\eta 2 \exp \bigg( \frac{i}{2} \, x^2 F '' (c) \bigg) \, dx + S 
\end{displaymath}
with 
\begin{displaymath}
|S| \leq \lambda_3 \int_{0}^\eta \frac{x^3}{3} \, dx = \frac{\, \eta^4}{12} \, \lambda_3 \, .
\end{displaymath}

Finally, set 
\begin{displaymath}
K = \int_{0}^\eta 2 \exp \bigg( \frac{i}{2} \, x^2 F '' (c) \bigg) \, dx \, .
\end{displaymath}
We make the change of variable $x = \sqrt{\frac{2}{F '' (c)}} \, \sqrt{t}$. Recall that 
$\int_{0}^\infty  \frac{\e^{i t}}{\sqrt{t}} dt = \sqrt{\pi} \, \e^{i \pi / 4}$ is the classical Fresnel integral, and that an 
integration by parts gives, for $m > 0$: 
\begin{displaymath} 
\bigg| \int_{m}^\infty  \frac{\e^{i t}}{\sqrt{t}} \, dt \bigg|  \leq \frac{2}{\sqrt m} \, \cdot
\end{displaymath}
Therefore, with $m = \frac{\, \eta^2}{2} \, F '' (c)$:
\begin{displaymath}
K = \sqrt{\frac{2}{F '' (c)}} \int_{0}^m  \frac{\e^{i t}}{\sqrt{t}} \, dt 
= \sqrt{\frac{2 \pi}{F '' (c)}} \, \e^{i \pi/4} + R_m \, ,
\end{displaymath}
with
\begin{displaymath}
|R_m| \leq C \, \sqrt{\frac{1}{F '' (c)}} \, \frac{1}{\sqrt{m}} \leq \frac{C}{\eta \lambda_2} \, \cdot
\end{displaymath}

All in all, we proved that 
\begin{equation} \label{aia} 
I_2 = \sqrt{\frac{2 \pi}{F '' (c)}} \, \exp \big[ i (F (c) + \pi / 4) \big] + O\, \bigg( \frac{1}{\eta \lambda_2} + \eta^4 \lambda_3 \bigg) \, \cdot
\end{equation}
and the same estimate holds for $I$, thanks to \eqref{ein} and \eqref{zwei}.
\smallskip

We have hence proved Proposition~\ref{tihebr}. 
\end{proof}
%

\section {Some results on multipliers} \label{sec: multipliers} 

The set $\mathcal{M} \big( H^2 (\beta) \big)$ of multipliers of $H^2 (\beta)$ is by definition the vector space of functions $h$ analytic on $\D$ and such that 
$h f \in H^2 (\beta)$ for all $f \in H^2 (\beta)$. When $h \in \mathcal{M} \big( H^2 (\beta) \big)$, the operator $M_h$ of multiplication by $h$ is bounded on 
$H^2 (\beta)$ by the closed graph theorem. The space $\mathcal{M} \big( H^2 (\beta) \big)$ equipped with the operator norm is a Banach space. We note the 
obvious property:
\begin{equation} \label{quinque} 
\mathcal{M} \big( H^2 (\beta) \big) \hookrightarrow H^\infty \quad \text{contractively.} 
\end{equation}
Indeed, if $h \in \mathcal{M} \big( H^2 (\beta) \big)$, we easily get for all $w \in \D$:
\begin{displaymath} 
M_{h}^{\ast} (K_w) = \overbar{h(w)} K_w \, ;
\end{displaymath} 
so that taking norms and simplifying, we are left with $|h (w)|\leq \| M_h \|$, showing that $h \in H^\infty$ with $\| h \|_\infty \leq  \| M_h \|$. 
\goodbreak

\begin{proposition} \label{prop multipliers}
We have ${\cal M} \big( H^2 (\beta) \big) = H^\infty$ isomorphically if and only if $\beta$ is essentially decreasing.
\end{proposition}
\begin{proof}
The sufficient condition is proved in \cite[beginning of the proof of Proposition~3.16]{LLQR-comparison}. For the necessity, we have 
$\| M_h \| \approx \| h \|_\infty$ for every $h \in H^\infty$. Now, for $m > n$ (recall that $e_n (z) = z^n$):
\begin{displaymath}
e_m (z) = z^{m - n} z^n = (M_{e_{m - n}} e_n) (z) \, ;
\end{displaymath}
so, since $\| M_{e_{m - n}} \| \leq C \, \| e_{m - n} \|_\infty = C$ for some positive constant $C$:
\begin{displaymath}
\beta_m = \| e_m \|^2 \leq C^2 \, \| e_n \|^2 = C^2 \, \beta_n \, . \qedhere
\end{displaymath}
\end{proof}

In \cite[Section~3.6]{LLQR-comparison}, we gave the following notion of \emph{admissible} Hilbert space of analytic functions.

\begin{definition}
A Hilbert space $H$ of analytic functions on $\D$, containing the constants, and with reproducing kernels $K_a$, $a \in \D$, is said 
\emph{admissible} if: 
\begin{enumerate}
\setlength\itemsep {-0.05 em}

\item [$(i)$]  $H^2$ is continuously embedded in $H$;  

\item [$(ii)$] $\mathcal{M} (H) = H^\infty$; 

\item [$(iii)$] the automorphisms of $\D$ induce bounded composition operators on $H$;

\item [$(iv)$] $\displaystyle \frac{\Vert K_a\Vert_H}{\Vert K_b\Vert_H} \leq h \bigg(\frac{1 - |b|}{1 - |a|} \bigg)$ for $a, b \in \D$, where 
$h \colon \R^+\to \R^+$ is an non-decreasing function.
\end{enumerate}
\end{definition}

We proved in that paper that every weighted Hilbert space $H^2 (\beta)$ with $\beta$ non-increasing is admissible, under the additional hypothesis 
that the automorphisms of $\D$ induce bounded composition operators. In view of Theorem~\ref{main theorem}, we get the following result.
\goodbreak

\begin{proposition} \label{prop admissible}
Let $\beta$ be essentially decreasing that satisfies the $\Delta_2$-condition. Then $H^2 (\beta)$ is admissible. 
\end{proposition}

Let us give a different proof.

\begin{proof}
Because $\beta$ is essentially decreasing, item $(i)$ holds, as well as item $(ii)$, by Proposition~\ref{prop multipliers}. Item $(iii)$ is 
Theorem~\ref{main theorem}. It remains to show $(iv)$. We may assume that $\beta$ is non-increasing.

Let $0 < s < r < 1$. 

Without loss of generality, we may assume that $r, s \geq 1 / 2$. It is enough to prove:
\begin{equation} \label{enough}
\| K_r \|^2 \leq C \, \| K_{r^2} \|^2
\end{equation}
for some constant $C > 1$. Indeed, iteration of \eqref{enough} gives:
\begin{displaymath}
\| K_r \|^2 \leq C^k \, \| K_{r^{2^k} } \|^2
\end{displaymath}
and if $k$ is the smallest integer such that $r^{2^k} \leq s$, we have $2^{k - 1} \log r > \log s$ and $2^k \leq D \, \frac{1 - s}{1 - r}$ 
where $D$ is a numerical constant. Writing $C = 2^\alpha$ with $\alpha > 1$, we obtain:
\begin{displaymath}
\bigg( \frac{\| K_r \|}{\| K_s \|} \bigg)^2 \leq C^k = (2^k)^\alpha \leq D^\alpha \bigg( \frac{1 - s}{1 - r} \bigg)^\alpha \, .
\end{displaymath}

To prove \eqref{enough}, we pick some $M > 1$ such that $\beta_{2 n} \geq M^{- 1} \beta_n$ for all $n \geq 1$ and write $t = r^2$. We have:
\begin{displaymath}
\| K_r \|^2 = \frac{1}{\beta_0} + \sum_{n = 1}^\infty \frac{t^{2 n}}{\beta_{2 n}} 
+ \sum_{n = 1}^\infty \frac{t^{2 n - 1}}{\beta_{2 n - 1}} \, \raise 1 pt \hbox{,}
\end{displaymath}
implying, since $\beta_{2 n - 1} \geq \beta_{2 n} \geq M^{- 1} \beta_n$ and $t^{2 n - 1} \leq 4 \, t^{2 n}$:
\begin{displaymath}
\| K_r \|^2 \leq \frac{1}{\beta_0} + M \sum_{n = 1}^\infty \frac{t^{2 n}}{\beta_n} 
+ 4 M \sum_{n = 1}^\infty \frac{t^{2 n}}{\beta_n} \leq 5 M \| K_t \|^2 \, . \qedhere
\end{displaymath}
\end{proof}

The notion of admissible Hilbert space $H$ is useful for the set of conditional multipliers:
\begin{displaymath} 
\mathcal{M} (H, \phi) = \{w \in H \tq w \, (f \circ \phi) \in H \text{ for all } f \in H \} \, . 
\end{displaymath} 
As a corollary of \cite[Theorem~3.18]{LLQR-comparison} we get:
\begin{corollary}
Let $\beta$ be essentially decreasing and satisfying the $\Delta_2$-condition. Then:
\begin{enumerate}
\setlength\itemsep {-0.05 em}

\item [$1)$] $\mathcal{M} (H^2, \phi) \subseteq \mathcal{M} \big(H^2 (\beta), \phi \big)$; 

\item [$2)$] $\mathcal{M} \big(H^2 (\beta), \phi \big) = H^2 (\beta)$ if and only if $\| \phi \|_\infty < 1$;
 
\item [$3)$] $\mathcal{M} \big(H^2 (\beta), \phi \big) = H^\infty$ if and only if $\phi$ is a finite Blaschke product.
\end{enumerate}
\end{corollary}
%

\section{Miscellaneous remarks} \label{sec: remarks}

Some of the results of this paper can slightly be improved.

\subsection{Conditions on the weight $\beta$} \label{cond. beta}

First, we say that a sequence $(\beta_n)$ of positive numbers is \emph{slowly oscillating} if there if a function $\rho \colon (0, \infty) \to (0, \infty)$ 
that is bounded on each compact subset of $(0, \infty)$ for which:
\begin{displaymath} 
\frac{\beta_m}{\beta_n} \leq \rho \, \bigg( \frac{m}{n} \bigg) \, \cdot
\end{displaymath} 

This clearly amounts to say that, for some positive constants $C_1 < C_2$, we have:
\begin{displaymath} 
\qquad\quad C_1 \leq \frac{\beta_m}{\beta_n} \leq C_2 \quad \text{when } n / 2 \leq m \leq 2 n \, .
\end{displaymath} 

Every essentially decreasing sequence with the $\Delta_2$-condition is slowly oscillating. 

\begin{proposition}
The following holds:
\smallskip

$1)$ every slowly oscillating sequence has a polynomial minoration;
\smallskip

$2)$ there are bounded sequences which are slowly oscillating, but not essentially decreasing. 
\end{proposition}
\begin{proof}
$1)$ is clear, because if $2^j \leq n < 2^{j + 1}$, then
\begin{displaymath} 
\beta_n \geq C^{- 1} \beta_{2^j} \geq C^{- j - 1} \beta_1 \geq C^{- 1} \beta_1 \, n^{- \alpha} \, ,
\end{displaymath} 
with $\alpha = \log C / \log 2$.

$2)$ We define $\beta_n$ as follows. Let $(a_k)$ be an increasing sequence of positive square integers such that 
$\lim_{k \to \infty} a_{k + 1} / a_k = \infty$, for example $a_k = 4^{k^2}$, and let $b_k = \sqrt{a_k a_{k + 1}}\,$; with our choice, this is an integer and 
we clearly have $a_k < b_k < a_{k + 1}$. We set:
\begin{align*}
\beta_n = 
\left\{ 
\begin{array} {ll}
\, a_k / n & \text{for } a_k \leq n < b_k \\
& \\
(a_k / b_k^2) \, n = (1 / a_{k + 1}) \, n & \text{for } b_k \leq n < a_{k + 1} \, .
\end{array}
\right.
\end{align*}
This sequence $(\beta_n)$ is slowly oscillating by construction. Indeed, it suffices to check that for $a_k \leq n / 2 < b_k \leq n < a_{k + 1}$, the quotient 
$\beta_m / \beta_n$ remains lower and upper bounded when $n / 2 \leq m \leq n$. But for $n / 2 \leq m < b_k$, we have
\begin{displaymath} 
\frac{\beta_m}{\beta_n} = \frac{a_k / m}{n / a_{k + 1}} = \frac{a_k a_{k + 1}}{m n} = \frac{b_k^2}{m n} \, \raise 1 pt \hbox{,}
\end{displaymath} 
which is $\leq 2 \, b_k^2 / n^2 \leq 2$ and $\geq b_k^2 / n^2 \geq (n / 2)^2 / n^2 = 1 / 4$; and for $b_k \leq m$, we have
\begin{displaymath} 
\frac{\beta_m}{\beta_n} = \frac{m / a_{k + 1}}{n / a_{k + 1}} = \frac{m}{n} \in [ 1 /2, 1] \, .
\end{displaymath} 

However, though $(\beta_n)$ is bounded, since $\beta_n \leq 1$ for $a_k \leq n < b_k$ and, 
for $b_k \leq n < a_{k + 1}$, 
\begin{displaymath} 
\beta_n \leq \beta_{a_{k + 1} - 1} = \frac{1}{a_{k + 1}}\, (a_{k + 1} - 1) \leq 1 \, ,
\end{displaymath} 
it is not essentially decreasing, since
\begin{displaymath} 
\frac{\beta_{a_{k + 1} - 1}}{\beta_{b_k}} = \frac{1}{\sqrt{a_k a_{k + 1}}} \, (a_{k + 1} - 1)  \sim \sqrt{\frac{a_{k + 1}}{a_k}} 
\converge_{k \to \infty} \infty \, . \qedhere
\end{displaymath} 
\end{proof}

By a slight modification (change the value of the constants in the definition of $\beta_n$), we could obtain a sequence which is slowly oscillating, tends to zero, 
yet again not essentially decreasing.
\smallskip

Now, Theorem~\ref{main theorem} admits the following variant.

\begin{theorem} \label{main bis}
Let $(\beta_n)$ be a sequence of positive numbers which is bounded above and slowly oscillating. Then all symbols that extend analytically in a neighborhood 
of $\overbar{\D}$ induce a bounded composition operator on $H^2 (\beta)$.
\end{theorem} 

It is the case, for example, for finite Blaschke products. The proof follows that of Theorem~\ref{main theorem}, with the help of the following lemma.
\begin{lemma} \label{lemma LU}
Let $(\beta_n)$ be a sequence of positive numbers which is bounded above and slowly oscillating. 
Let $A = (a_{m, n})_{m, n}$ be the matrix of a bounded operator on $\ell_2$. Assume that, for constants $C_1 < 1$, $C_2 > 1$, and $c$, $b$, we have:
\smallskip

$1)$ $|a_{m, n}| \leq c \, \e^{- b n}$ for $m \leq C_1 n$;
\smallskip

$2)$ $|a_{m, n}| \leq c \, \e^{- b m}$ for $m \geq C_2 n$.

Then the matrix $\dis \tilde A = \bigg( a_{m, n} \, \sqrt{\frac{\beta_m}{\beta_n}} \bigg)$ also defines a bounded operator on $\ell_2$. 
\end{lemma}
\goodbreak
\begin{proof} [Sketch of proof] 
The matrix $\tilde A$ is Hilbert-Schmidt far from the diagonal since, with $ \lambda \geq \beta_n \geq \delta \, n^{ - \alpha}$, we have:
\begin{displaymath} 
\sum_{m < C_1 n} |a_{m, n}|^2 \beta_m / \beta_n \leq \sum_{m < C_1 n} \lambda \, \delta^{ - 1} n^\alpha |a_{m, n}|^2 
\lesssim \sum_{n \geq 1} n^{\alpha + 1} \e^{- b n} < \infty \, ,
\end{displaymath} 
and
\begin{displaymath} 
\sum_{m > C_2 n} |a_{m, n}|^2 \beta_m / \beta_n \leq \sum_{m > C_2 n} \lambda \, \delta^{ - 1} n^\alpha |a_{m, n}|^2 
\lesssim \sum_{n \geq 1} n^\alpha \bigg( \sum_{m > C_2 n} \e^{- b m} \bigg) \, . 
\end{displaymath} 
Since  $\beta_m / \beta_n$ remains bounded from above and below around the diagonal, the matrix $\tilde A$ behaves like $A$ near the diagonal.
\end{proof}

\noindent {\bf Remark.} The proof shows that, instead of $1)$ and $2)$, it is enough to have:
\begin{displaymath} 
\sum_{m < C_1 n} n^{\alpha + 1} |a_{m, n}|^2 < \infty \quad \text{and} \quad \sum_{m > C_2 n} m^\alpha |a_{m, n}|^2 < \infty \, .
\end{displaymath} 
Moreover the proof also shows that when $\beta$ is slowly oscillating, if we set $E = \{ (m, n) \tq C_1 n \leq m \leq C_2 n\}$, then the 
matrix $\big( \sqrt{\beta_m / \beta_n} \, \ind_E (m, n) \big)$ is a Schur multiplier over \emph{all} the bounded matrices, while Kacnel'son's theorem 
(Theorem~\ref{theo Kacnelson}) says that, if $\gamma = (\gamma_n)$ is non-increasing, the matrix $(\gamma_m / \gamma_n)$ is a Schur 
multiplier of all bounded \emph{lower-triangular} matrices. 

\begin{proof} [Proof of Theorem \ref{main bis}]
We first prove that the assumptions of Lemma \ref{lemma LU} are satisfied with $a_{m,n}=\widehat{\varphi^{n}}(m)$. This works  nearly as in 
Lemma~\ref{majo coeff}. 
For every symbol $\phi$, let $M (r) = \sup_{|z| = r} |\phi (z)|$ and write $M (1 / \e) = \e^{- \delta}$, with $\delta > 0$. 
Let $A$ be the matrix of $C_\phi$, with respect to the canonical basis of $H^2$. 
The Cauchy inequalities give, if $m \leq (\delta / 2) n$:
\begin{displaymath} 
|a_{m, n} | \leq [M (1 / \e)]^n \e^m \leq \e^{m - \delta n} \leq \e^{- \delta / 2 n}
\end{displaymath} 

When $\phi$ is analytic in a neighborhood of $\overbar {D (0, R)}$ with $R > 1$, the Cauchy inequalities give, writing $M (R) = \e^\rho$ and 
$R = \e^\delta$, for $m \geq C_2 n$, and a suitable constant $C_2$ (take $C_2 = (2\rho) / \delta$ for instance):
\begin{displaymath} 
|a_{m, n} | \leq \frac{[M (R)]^n}{R^m} \leq \e^{n \rho  - \delta m} \leq \e^{- (\delta /2) m } \, . 
\end{displaymath} 
We now conclude with Lemma~\ref{lemma LU}. 
\end{proof}
%

\subsection{Singular inner functions} \label{inner}

Let $a > 0$ and let $I_a$ be the singular inner function defined by 
\begin{equation} \label{a} 
I_{a} (z) = \exp\bigg( - a \, \frac{1 + z}{1 - z}\bigg) = \sum_{m = 0}^\infty c_{m} (a) z^m \, .
\end{equation}

If $n$ is a positive integer, we have $I_{n} (z) = [I_{1} (z)]^n$ and we write 
\begin{displaymath} 
I_{n} (z) = \sum_{m = 0}^\infty a_{m, n} z^m \, , \quad \text{with } a_{m, n} = c_{m} (n) \, .
\end{displaymath} 

We rely on the following lemma, familiar to  experts in orthogonal polynomials and special functions, but maybe not  so much as regards the uniformity, 
essential for our present purposes (see \cite{NS} or \cite{L}).  
\begin{lemma} \label{cruc2} 
It holds
\begin{equation} \label{coeff} 
c_{m} (n) = c \, n^{1/4} m^{- 3/4} \cos ( 2 \, \sqrt{2 \, n m}+\pi/4) + R_{m} (n) =: M_m (n) + R_{m}(n) \, ,
\end{equation} 
where $c = \pi^{- 1 / 2} 2^{1 / 4}$ and where  $|R_{m}(n)| \leq K \sqrt n \,  m^{- 5/4}$, with $K$ some numerical constant. 
\end{lemma}

We use  the following \cite[ p.~253]{NS} and \cite[p.~198]{SZ}, where $ L_{m}^{(\alpha)}$ denotes the generalized Laguerre polynomial of degree $m$ with 
parameter $\alpha$.

\begin{theorem} \label{hong}  
With the notation of \eqref{a}, we have
\begin{displaymath} 
c_{m} (a) = \e^{- a} \, L_{m}^{(-1)} (2 a) \, .
\end{displaymath} 

Moreover, we have, uniformly for $0 < \eps \leq x \leq M < \infty$:
\begin{displaymath} 
L_{m}^{(\alpha)} (x) = \pi^{- 1/2} \e^{x/2} x^{- \alpha/2 - 1 / 4} m^{\alpha/2 - 1/4} \cos \big( 2 \sqrt{m x} - \alpha \pi/2 - \pi/4 \big) + R_m (x)\, ,
\end{displaymath} 
where $|R_m (x)|\leq K_\alpha \, m^{\alpha/2 - 3/4}$,  and $K_\alpha > 0$ only depending on $\alpha$.
\end{theorem} 

Now, using Theorem~\ref{hong} with $\alpha = - 1$ and $a = n$, we get Lemma~\ref{cruc2}. 
Actually, to get Lemma~\ref{cruc2}, we need some uniformity with respect to $a$ in Theorem~\ref{hong}, which is not given by the above statement of 
that theorem; but provided we change $m^{- 5/4}$ into $\sqrt{n} \, m^{- 5/4}$ in $R_m$, a careful examination of Fej\'er's proof of Theorem~\ref{hong} 
shows that this uniformity holds. Alternatively, write $c_{m} (a)$ as a Fourier coefficient
\begin{equation} \label{foco} 
c_{m} (n) = \frac{1}{\pi} \int_{- \pi/2}^{\pi/2} \exp \big[ i (n \cot x + 2 \, m x) \big] \, dx 
\end{equation}
and use the previous  van der Corput estimates.
\smallskip

Once this lemma is at our disposal, we can prove again the following theorem.

\begin{theorem} \label{alter} 
Assume that $\beta$ is a non-increasing sequence and that the composition operator $C_{I_1}$ maps $H^{2} (\beta)$ to itself. Then $\beta$ satisfies the 
$\Delta_2$-condition. 
\end{theorem}
\begin{proof}
Our assumption implies, with $M = \| C_{I_1} \|^2$:
\begin{equation} \label{imply} 
\sum_{m = 0}^\infty |a_{m, n}|^2 \beta_m \leq M \beta_n \, .
\end{equation}
We use Lemma~\ref{cruc2} with $C^{- 2} n \leq m < n$, where $C > 1$ satisfies $C \sqrt{2} < \pi / 2$, which is possible since $2 \sqrt{2} < \pi$. 
The term $R_m (n)$ is dominated by $K C^{- 1/2} m^{- 3/4}$.  Let $\theta_m = 2 \sqrt{2} \, \sqrt{n m} + \pi / 4$ be the argument of the number appearing 
in \eqref{coeff}. We have 
\begin{displaymath} 
\theta_{m + 1} - \theta_m = \frac{2 \sqrt{2} \, \sqrt{n}}{\sqrt m + \sqrt{m + 1}} \, ; 
\end{displaymath} 
hence
\begin{displaymath} 
\sqrt{2} \leq \theta_{m + 1} - \theta_m \leq C \, \sqrt{2} < \frac{\pi}{2} \, \cdot
\end{displaymath} 
The argument $\theta_m$ then varies regularly, 
and there is a positive constant $\delta$ and a subset $E$ of integers in the interval $[C^{- 2} n, n]$ such that 
$|E| \geq \delta \, n$ and 
\begin{displaymath} 
M_m (n) \geq 2 \, \delta \, n^{1/4} m^{- 3/4} \geq 2 \, \delta \, n^{- 1/2} 
\end{displaymath} 
for all $m \in E$. Therefore, for $n$ large enough, we have, for all $m \in E$:
\begin{displaymath} 
|a_{m, n}| \geq  2 \, \delta \, n^{- 1/2} - K \sqrt{n} \, \, m^{- 5/4} 
\geq 2 \, \delta \, n^{- 1/2}  - K C^{5/2} n^{- 3 / 4} \geq \delta \, n^{- 1 / 2} \, . 
\end{displaymath} 
With this information, \eqref{imply} gives:
\begin{displaymath} 
M \, \beta_n \geq \sum_{m \in E} |a_{m, n}|^2 \beta_m \geq \delta^2 n^{- 1} |E| \, \beta_{\lfloor C^{- 2} n \rfloor} 
\geq \delta^3 \beta_{\lfloor C^{- 2} n \rfloor} \, ,
\end{displaymath} 
where $\lfloor \, . \, \rfloor$ stands for the integer part, and this proves the theorem. 
\end{proof}
%

\bigskip

\noindent{\bf Acknowledgements.} We warmly thank R.~Zarouf for useful discussions and informations. 

L. Rodr{\'\i}guez-Piazza is partially supported by the project PGC2018-094215-B-I00 
(Spanish Ministerio de Ciencia, Innovaci\'on y Universidades, and FEDER funds). 
Parts of this paper was made when he visited the Universit\'e d'Artois in Lens and the Universit\'e de Lille in January 2020. It is his pleasure to 
thank all his colleagues in these universities for their warm welcome.

The third-named author was partly supported by the Labex CEMPI (ANR-LABX-0007-01). 

This work is also partially supported by the grant ANR-17-CE40-0021 of the French National Research Agency ANR (project Front).

\goodbreak

\smallskip

{\footnotesize
Pascal Lef\`evre \\
Univ. Artois, Laboratoire de Math\'ematiques de Lens (LML) UR~2462, \& F\'ed\'eration Math\'ematique des Hauts-de-France FR~2037 CNRS, 
Facult\'e Jean Perrin, Rue Jean Souvraz, S.P.\kern 1mm 18 
F-62\kern 1mm 300 LENS, FRANCE \\
pascal.lefevre@univ-artois.fr
\smallskip

Daniel Li \\ 
Univ. Artois, Laboratoire de Math\'ematiques de Lens (LML) UR~2462, \& F\'ed\'eration Math\'ematique des Hauts-de-France FR~2037 CNRS, 
Facult\'e Jean Perrin, Rue Jean Souvraz, S.P.\kern 1mm 18 
F-62\kern 1mm 300 LENS, FRANCE \\
daniel.li@univ-artois.fr
\smallskip

Herv\'e Queff\'elec \\
Univ. Lille Nord de France, USTL,  
Laboratoire Paul Painlev\'e U.M.R. CNRS 8524 \& F\'ed\'eration Math\'ematique des Hauts-de-France FR~2037 CNRS, 
F-59\kern 1mm 655 VILLENEUVE D'ASCQ Cedex, FRANCE \\
Herve.Queffelec@univ-lille.fr
\smallskip
 
Luis Rodr{\'\i}guez-Piazza \\
Universidad de Sevilla, Facultad de Matem\'aticas, Departamento de An\'alisis Matem\'atico \& IMUS,  
Calle Tarfia s/n  
41\kern 1mm 012 SEVILLA, SPAIN \\
piazza@us.es
}

\end{document}